\documentclass[11pt,a4paper]{article}

\usepackage{pgf,tikz,tkz-graph}
\usepackage{csquotes}
\usepackage[english]{babel}

\usepackage{mathrsfs}
\usetikzlibrary{arrows}
\usetikzlibrary{arrows.meta}
\usepackage[T1]{fontenc}
\usepackage{amsfonts}
\usepackage{amssymb}
\usepackage{amsthm}
\usepackage{amscd}
\usepackage{graphicx}
\usepackage{enumitem}
\usepackage{authblk}
\usepackage{verbatim}
\usepackage{hyperref}
\usepackage{amsmath,caption}
\usepackage{url,pdfpages,xcolor,framed,color}
\usepackage{a4wide}

\usepackage{array}
\theoremstyle{plain}

\newtheorem{thr}{Theorem}[section]

\newtheorem{lem}[thr]{Lemma}
\newtheorem{prop}[thr]{Proposition}
\newtheorem{conj}[thr]{Conjecture}
\theoremstyle{definition}
\newtheorem{defi}[thr]{Definition}

\newtheorem{remark}[thr]{Remark}

\newtheorem{prob}[thr]{Problem}
\newtheorem{claim}[thr]{Claim}

\def\vc{\overrightarrow}

\DeclareMathOperator{\ecc}{ecc}
\DeclareMathOperator{\rad}{rad}

\newcommand*{\myproofname}{Proof}
\newenvironment{claimproof}[1][\myproofname]{\begin{proof}[#1]}{\end{proof}}

\title{Maximum size of digraphs of given radius}
\author{Stijn Cambie\footnote{Department of Mathematics, Radboud University Nijmegen, Postbus 9010, 6500 GL Nijmegen, The Netherlands. Email: \href{mailto:stijn.cambie@hotmail.com}{stijn.cambie@hotmail.com}. This work has been supported by a Vidi Grant of the Netherlands Organization for Scientific Research (NWO), grant number $639.032.614$.} }%
\date{}

\begin{document}
	\definecolor{xdxdff}{rgb}{0.49019607843137253,0.49019607843137253,1.}
	\definecolor{ududff}{rgb}{0.30196078431372547,0.30196078431372547,1.}
	
	\tikzstyle{every node}=[circle, draw, fill=black!50,
	inner sep=0pt, minimum width=4pt]

	\maketitle

	\begin{abstract}
		In $1967$, Vizing determined the maximum size of a graph with given order and radius. In $1973$, Fridman answered the same question for digraphs with given order and outradius. We investigate that question when restricting to biconnected digraphs. Biconnected digraphs are the digraphs with a finite total distance and hence the interesting ones, as we want to note a connection between minimizing the total distance and maximizing the size under the same constraints. 
		We characterize the extremal digraphs maximizing the size among all biconnected digraphs of order $n$ and outradius $3$, as well as when the order is sufficiently large compared to the outradius. As such, we solve a problem of Dankelmann asymptotically.
		We also consider these questions for bipartite digraphs and solve a second problem of Dankelmann 
		partially.
	\end{abstract}

	\section{Foreword}\label{sec:foreword}
	The order $\lvert V \rvert$ and size $\lvert E \rvert$ or $\lvert A \rvert$ are without any doubt the most basic values associated to a graph or digraph.
	The next natural parameters associated with a graph or digraph are degree-based or distance-based.
	Among the distance-based parameters, the total distance, diameter and radius are the first that pop up.
	As such, questions about the relationships between these quantities can be considered as some of the most fundamental questions in graph theory.
	We will focus on the questions about the maximum size, given order and radius for digraphs. In the latter case, we mostly focus on the outradius.	
	Hereby we give conjectures and asymptotic results answering some questions of Dankelmann~\cite{D15} for biconnected digraphs and biconnected bipartite digraphs.
	In Table~\ref{tbl:overview}, we give an overview of the (possible conjecturally) extremal graphs and digraphs attaining the maximum size given radius or outradius.
	We do this for both connected and biconnected digraphs, and consider both the general and the bipartite case. 
	
	The family of bipartite, biconnected digraphs seem to be special in the sense that it is the only considered family of digraph for which the extremal digraph is unique, if the corresponding Conjecture~\ref{conj:extr_bicon_bip_digraph_rad+ge4} is correct. This behaviour has been proven for even radius and sufficiently large order. 
	From a different angle, we show that the graphs minimizing the distance and those maximizing the size under the same conditions are somewhat related in many cases.
	This is true when $n$ is sufficiently large in terms of the radius, when considering both graphs as well as digraphs given outradius, but remarkably that is not the case for digraphs given radius.

	\begin{center}

	\begin{table}[h!]
		\centering
		\begin{tabular}{ | m{1.5cm} | m{7cm} | m{7cm} | }
			\hline
			 & \begin{minipage}[t]{7cm} \centering General 	\end{minipage} & \begin{minipage}[t]{7cm}\centering Bipartite 	\end{minipage}\\ 
			 \hline
			 &&\\
			Graphs
			&
			\begin{minipage}[t]{7cm}
				\centering
				\scalebox{0.95}{
				\begin{tikzpicture}
						\definecolor{cv0}{rgb}{0.0,0.0,0.0}
						\definecolor{c}{rgb}{1.0,1.0,1.0}
						
						\Vertex[L=\hbox{$K_{s}$},x=1cm,y=-1cm]{v0}	
						\Vertex[L=\hbox{$K_{n-2r+s}$},x=1,y=1]{v1}
						\Vertex[L=\hbox{$v_2$},x=3,y=2]{v2}
						\Vertex[L=\hbox{$v_3$},x=5,y=2]{v3}
						\Vertex[L=\hbox{$v_4$},x=6.5,y=1]{v4}
						\Vertex[L=\hbox{$v_5$},x=6.5,y=-1]{v5}
						\Vertex[L=\hbox{$v_{2r-1}$},x=3,y=-2]{v7}
						\Vertex[L=\hbox{$v_{2r-2}$},x=5,y=-2]{v6}	
						
						\Edge[lw=0.1cm,style={}](v0)(v1)
						\Edge[lw=0.1cm,style={}](v1)(v2)
						\Edge[lw=0.1cm,style={}](v2)(v3)
						\Edge[lw=0.1cm,style={}](v3)(v4)
						\Edge[lw=0.1cm,style={}](v4)(v5)	
						\Edge[lw=0.1cm,style=dotted](v5)(v6)	
						\Edge[lw=0.1cm,style={}](v6)(v7)	
						\Edge[lw=0.1cm,style={}](v7)(v0)
						\end{tikzpicture}
				}
						
			\end{minipage}
			& 
			\begin{minipage}[t]{7cm}
				\centering
				\scalebox{0.95}{
					\begin{tikzpicture}
					\definecolor{cv0}{rgb}{0.0,0.0,0.0}
					\definecolor{c}{rgb}{1.0,1.0,1.0}
					
					\Vertex[L=\hbox{a$K_{1}$},x=1cm,y=-1cm]{v0}	
					\Vertex[L=\hbox{b$K_{1}$},x=1,y=1]{v1}
					\Vertex[L=\hbox{c$K_1$},x=3,y=2]{v2}
					\Vertex[L=\hbox{$v_3$},x=5,y=2]{v3}
					\Vertex[L=\hbox{$v_4$},x=6.5,y=1]{v4}
					\Vertex[L=\hbox{$v_5$},x=6.5,y=-1]{v5}
					\Vertex[L=\hbox{$v_{2r-1}$},x=3,y=-2]{v7}
					\Vertex[L=\hbox{$v_{2r-2}$},x=5,y=-2]{v6}	
					
					\Edge[lw=0.1cm,style={}](v0)(v1)
					\Edge[lw=0.1cm,style={}](v1)(v2)
					\Edge[lw=0.1cm,style={}](v2)(v3)
					\Edge[lw=0.1cm,style={}](v3)(v4)
					\Edge[lw=0.1cm,style={}](v4)(v5)	
					\Edge[lw=0.1cm,style=dotted](v5)(v6)	
					\Edge[lw=0.1cm,style={}](v6)(v7)	
					\Edge[lw=0.1cm,style={}](v7)(v0)
					\end{tikzpicture}
				}
							
			\end{minipage}\\

		&	\begin{minipage}[t]{7cm} \centering $G_{n,r,s}$\\
		\cite{VZ67}	\end{minipage}&\begin{minipage}[t]{7cm} \centering blow-up $C_{2r}$ with stable sets\\	
		\cite{DSv11}\end{minipage}\\
		&&\\
		 \hline
			&&\\
			Digraphs
			&
			\begin{minipage}[t]{7cm}
				\centering
				\scalebox{0.9}{
				\begin{tikzpicture}
				
				\definecolor{cv0}{rgb}{0.0,0.0,0.0}
				\definecolor{c}{rgb}{1.0,1.0,1.0}

				\Vertex[L=\hbox{$v_0$},x=2,y=-2.5]{v0}
				\Vertex[L=\hbox{$K_{s}$},x=2,y=-1]{v1}
				\Vertex[L=\hbox{$K_{t}$},x=3.5,y=-1]{v2}
				\Vertex[L=\hbox{$v_3$},x=5,y=-1]{v3}
				\Vertex[L=\hbox{$v_4$},x=6.5,y=-1]{v4}
				\Vertex[L=\hbox{$v_r$},x=8,y=-1]{v5}

				\Edge[lw=0.1cm,style={post, right}](v0)(v1)
				\Edge[lw=0.1cm,style={post, right}](v2)(v1)
				\Edge[lw=0.1cm,style={post, right}](v3)(v2)
				\Edge[lw=0.1cm,style={post, right}](v4)(v3)

				\Edge[lw=0.1cm,style={post, right}](v1)(v2)
				\Edge[lw=0.1cm,style={post, right}](v2)(v3)
				\Edge[lw=0.1cm,style={post, right}](v3)(v4)
				\Edge[lw=0.05cm,style={dotted, right}](v4)(v5)
				
				\Edge[lw=0.1cm,style={post, bend left}](v3)(v1)
				\Edge[lw=0.1cm,style={post, bend left}](v4)(v1)	
				\Edge[lw=0.1cm,style={post, bend right}](v5)(v1)		
				\Edge[lw=0.1cm,style={post, bend left}](v4)(v2)	
				\Edge[lw=0.1cm,style={post, bend right}](v5)(v2)	
				\Edge[lw=0.1cm,style={post, bend right}](v5)(v3)	
				\Edge[lw=0.1cm,style={post, bend right}](v5)(v4)		
				\end{tikzpicture}}
					
			\end{minipage}
			& 
			\begin{minipage}[t]{7cm}
				
			\centering	
		\begin{tikzpicture}
		
		\definecolor{cv0}{rgb}{0.0,0.0,0.0}
		\definecolor{c}{rgb}{1.0,1.0,1.0}

		\Vertex[L=\hbox{$v_0$},x=2,y=-1]{v0}
		\Vertex[L=\hbox{$v_1$},x=3,y=1]{v1}
		\Vertex[L=\hbox{$v_2$},x=4,y=-1]{v2}
		\Vertex[L=\hbox{$a K_1$},x=5,y=1]{v3}
	
		\Vertex[L=\hbox{$b K_1$},x=6,y=-1]{v4}
		\Vertex[L=\hbox{$c K_1$},x=7,y=1]{v5}
		\Vertex[L=\hbox{$v_r$},x=8,y=-1]{v6}
	
		\Edge[lw=0.1cm,style={post, right}](v2)(v1)
		\Edge[lw=0.1cm,style={post, right}](v3)(v2)
		\Edge[lw=0.1cm,style={post, right}](v4)(v3)
		\Edge[lw=0.1cm,style={post, right}](v5)(v4)
		\Edge[lw=0.1cm,style={post, right}](v6)(v5)
		\Edge[lw=0.05cm,style={post, right}](v5)(v2)

		\Edge[lw=0.1cm,style={post, right}](v0)(v1)
		\Edge[lw=0.1cm,style={post, right}](v1)(v2)
		\Edge[lw=0.1cm,style={post, right}](v2)(v3)
		\Edge[lw=0.1cm,style={post, right}](v3)(v4)
		\Edge[lw=0.1cm,style={post, right}](v4)(v5)
		\Edge[lw=0.1cm,style={post, right}](v5)(v6)
		
		\Edge[lw=0.1cm,style={post, right}](v4)(v1)
		\Edge[lw=0.1cm,style={post, right}](v6)(v1)
		\Edge[lw=0.1cm,style={post, right}](v6)(v3)

		\end{tikzpicture}
						
			\end{minipage}\\
		&\begin{minipage}[t]{7cm} \centering
			$\overline{\Gamma}^\star_{n,r,1,s}$	\\ \cite{F}
		\end{minipage}	&\begin{minipage}[t]{7cm} \centering blow-up $\overline{\Gamma}^\star_{r+1} \cap K_{ \lceil \frac{r+1}2 \rceil, \lfloor \frac{r+1}2 \rfloor }$ with stable sets\\
		Section~\ref{sec:bipdigraph_maxSize}\end{minipage} \\
		&&\\
			 \hline
			&&\\
			\begin{minipage}[t]{1.6cm} Biconn. Digraphs\end{minipage}
			&
			\begin{minipage}[t]{7cm}
						\centering
					\scalebox{0.75}{
						\begin{tikzpicture}
						\definecolor{cv0}{rgb}{0.0,0.0,0.0}
						\definecolor{c}{rgb}{1.0,1.0,1.0}

						\Vertex[L=\hbox{$v_2$},x=9.5,y=0]{v1}
						\Vertex[L=\hbox{$v_r$},x=11.5,y=0]{v2}
						
						\Vertex[L=\hbox{$w_2$},x=5,y=0]{w1}
						\Vertex[L=\hbox{$w_r$},x=3,y=0]{w2}
						\Vertex[L=\hbox{$K_s$},x=6.5,y=0]{w0}
						
						\Vertex[L=\hbox{$K_{t}$},x=8cm,y=0.0cm]{v0}
						
						\Edge[lw=0.1cm,style={post, right}](v0)(v1)
						\Edge[lw=0.05cm,style={post, dotted, right}](v1)(v2)
						\Edge[lw=0.1cm,style={post, right}](w0)(v0)
						\Edge[lw=0.1cm,style={post, right}](v1)(v0)
						\Edge[lw=0.05cm,style={post, dotted, right}](v2)(v1)
						\Edge[lw=0.1cm,style={post, right}](v0)(w0)
						\Edge[lw=0.05cm,style={post, dotted, right}](w1)(w2)
						\Edge[lw=0.1cm,style={post, right}](w0)(w1)
						\Edge[lw=0.1cm,style={post, right}](w1)(w0)
						\Edge[lw=0.05cm,style={post, dotted, right}](w2)(w1)

						\Edge[lw=0.1cm,style={post, bend left}](w2)(w0)
						\Edge[lw=0.1cm,style={post, bend right}](w2)(v0)
						\Edge[lw=0.1cm,style={post, bend right}](w1)(v0)
						
						\Edge[lw=0.1cm,style={post, bend right=90}](v2)(w0)
						\Edge[lw=0.1cm,style={post, bend left}](v2)(v0)
						\Edge[lw=0.1cm,style={post, bend right=90}](v1)(w0)

						\end{tikzpicture}}					
			\end{minipage}
			& 
			\begin{minipage}[t]{7cm}
				\centering
			\scalebox{0.6}{
				\begin{tikzpicture}
				\definecolor{cv0}{rgb}{0.0,0.0,0.0}
				\definecolor{c}{rgb}{1.0,1.0,1.0}

				\Vertex[L=\hbox{$v_2$},x=9.5,y=0]{v1}
				\Vertex[L=\hbox{$v_3$},x=11,y=0]{v2}
				\Vertex[L=\hbox{$v_r$},x=12.5,y=0]{v3}
				\Vertex[L=\hbox{$w_2$},x=4.5,y=0]{w1}
				\Vertex[L=\hbox{$w_3$},x=3,y=0]{w2}
				\Vertex[L=\hbox{$w_r$},x=1.5,y=0]{w3}
				\Vertex[L=\hbox{$\lfloor a \rfloor K_1$},x=6,y=0]{w0}
				
				\Vertex[L=\hbox{$\lceil a \rceil K_1$},x=8cm,y=0.0cm]{v0}
				
				\Edge[lw=0.1cm,style={post, right}](v0)(v1)
				\Edge[lw=0.1cm,style={post, right}](v1)(v2)
				\Edge[lw=0.05cm,style={post, dotted, right}](v3)(v2)
				\Edge[lw=0.1cm,style={post, right}](w0)(v0)
				\Edge[lw=0.1cm,style={post, right}](v1)(v0)
				\Edge[lw=0.1cm,style={post, right}](v2)(v1)
				\Edge[lw=0.05cm,style={post, dotted, right}](v2)(v3)
				\Edge[lw=0.1cm,style={post, right}](v0)(w0)
				\Edge[lw=0.1cm,style={post, right}](w1)(w2)
				\Edge[lw=0.05cm,style={post, dotted, right}](w3)(w2)
				\Edge[lw=0.1cm,style={post, right}](w0)(w1)
				\Edge[lw=0.1cm,style={post, right}](w1)(w0)
				\Edge[lw=0.1cm,style={post, right}](w2)(w1)
				\Edge[lw=0.05cm,style={post, dotted, right}](w2)(w3)	
			
				\Edge[lw=0.1cm,style={post, bend right}](w2)(v0)
				\Edge[lw=0.1cm,style={post, bend left}](w3)(w0)

				\Edge[lw=0.1cm,style={post, bend right}](v2)(w0)
				
				\Edge[lw=0.1cm,style={post, bend left}](v3)(v0)

				\end{tikzpicture}}
					
			\end{minipage}\\
		&\begin{minipage}[t]{7cm} \centering $D_{n,r,s}$\\	
			Section~\ref{sec:Vizing_digraph} \end{minipage}&\begin{minipage}[t]{7cm}  \centering $D_{n,r,\lfloor \frac{n}{2}-r+1 \rfloor} \cap K_{ \lfloor \frac{n}{2} \rfloor, \lceil \frac{n}{2} \rceil}$	\\
			Section~\ref{sec:bipbicDigraphs_maxSize}	 \end{minipage}\\
		&&\\
		 \hline
		\end{tabular}
		\caption{Overview of extremal graphs}\label{tbl:overview}
	\end{table}
	
\end{center}

	\section{Introduction}\label{sec:intro}
	Since order, size and radius are all of the most natural parameters of a graph or digraph, it is not surprising that we can start with an exploration of some history that originated more than $50$ years ago.
	Here we give an overview about the maximum size of graphs and digraphs given their order and radius (outradius), as well as the new conjectures and partial results for biconnected digraphs. In each of the three cases, we also consider the bipartite case.
	
	People less familiar with some terminology can first go through Section~\ref{not&def} if necessary.

	In $1967$, Vadim Vizing~\cite{VZ67} determined the maximum size among all graphs of given order and radius. From his proof, one can conclude that the extremal graphs are exactly the graphs of the form $G_{n,r,s}$, when $r\ge 3$. Here $G_{n,r,s}$ is a cycle $C_{2r}$ in which we take blow-ups in $2$ consecutive vertices by cliques $K_s$ and $K_{n-2r+2-s}$, as defined in Section~\ref{not&def}. In Figure~\ref{fig:Gnrs} an example is presented when $r=4$.

	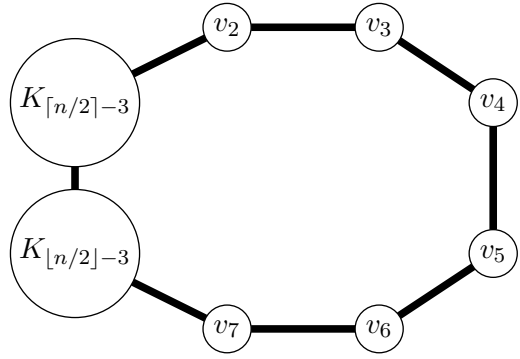
\begin{figure}[h]
		\centering
		\begin{tikzpicture}
		\definecolor{cv0}{rgb}{0.0,0.0,0.0}
		\definecolor{c}{rgb}{1.0,1.0,1.0}
		
		\Vertex[L=\hbox{$K_{\lfloor n/2 \rfloor-3}$},x=1cm,y=-1cm]{v0}	
		\Vertex[L=\hbox{$K_{\lceil n/2 \rceil-3}$},x=1,y=1]{v1}
		\Vertex[L=\hbox{$v_2$},x=3,y=2]{v2}
		\Vertex[L=\hbox{$v_3$},x=5,y=2]{v3}
		\Vertex[L=\hbox{$v_4$},x=6.5,y=1]{v4}
		\Vertex[L=\hbox{$v_5$},x=6.5,y=-1]{v5}
		\Vertex[L=\hbox{$v_7$},x=3,y=-2]{v7}
		\Vertex[L=\hbox{$v_6$},x=5,y=-2]{v6}	
		
		\Edge[lw=0.1cm,style={}](v0)(v1)
		\Edge[lw=0.1cm,style={}](v1)(v2)
		\Edge[lw=0.1cm,style={}](v2)(v3)
		\Edge[lw=0.1cm,style={}](v3)(v4)
		\Edge[lw=0.1cm,style={}](v4)(v5)	
		\Edge[lw=0.1cm,style={}](v5)(v6)	
		\Edge[lw=0.1cm,style={}](v6)(v7)	
		\Edge[lw=0.1cm,style={}](v7)(v0)
		\end{tikzpicture}
		\caption{The graph $G_{n,4,\lfloor n/2 \rfloor-3}$, a graph with radius $4$ maximizing size}\label{fig:Gnrs}
	\end{figure}
	
	\begin{thr}[\cite{VZ67}]\label{VZ}
		Let $f(n,r)$ be the maximum of edges in a graph with radius $r$.
		Then $f(n,1)=\binom{n}2, f(n,2)= \left \lfloor \frac{n(n-2)}2 \right \rfloor$ and 
		$$f(n,r)=\frac{(n-2r)^2+5n-6r}{2} \mbox{ when } n \ge 2r \ge 6.$$
		Equality occurs if and only if $G$ is a complete graph when $r=1$ or a complete graph minus a maximum matching (and an additional edge covering the remaining vertex when $n$ is odd) if $r=2.$
		When $r \ge 3$, equality occurs if and only if $G$ is isomorphic to $G_{n,r,s}$ where $1 \le s \le \frac{n-2r+2}{2}$.
	\end{thr}
	
	The analogous question about the minimum size is trivial in the sense that it is always $n-1$ but has also been considered with e.g. a minimum degree condition in~\cite{DV}.
	The maximization problem has been considered under some additional constraints as well, for example when restricting to bipartite graphs. This has been done in~\cite[Thr. 2]{DSv11}. 
	Note that $G_{n,r,s}$ would become bipartite if one would remove the edges in the cliques $K_s$ and $K_{n-2r+2-s}$ in the extremal graph $G_{n,r,s}.$ It turns out that this gives an extremal graph when performing this on $G_{n,4,\lfloor n/2 \rfloor-3}$, but the characterization of the extremal graphs is different. We state the result only for the case where $r \ge 4$ (see \cite[Thr.2]{DSv11} for the cases with $r \le 3$).
	\begin{thr}[\cite{DSv11}]\label{thr:DSv11}
		When $r \ge 4$, a bipartite graph $D$ with order $n$ and radius $r$ has at most $\lfloor \left(\frac{n}{2}-r+2\right)^2 \rfloor +2r-4$ edges.
		The extremal graph is obtained by taking blow-ups in $3$ consecutive vertices of a cycle $C_{2r}$ with independent sets of order $a,b$ and $c$ respectively where $a+b+c=n-2r+3$ and $0\le\lvert a+c-(b+1) \rvert \le 1.$ 
	\end{thr}
	Note that the size is indeed $2r-4+(b+1)(a+c)=\lfloor \left(\frac{n}{2}-r+2\right)^2 \rfloor +2r-4.$
	As an example, the extremal graphs for $n=19$ and $r=4$ are exactly those presented in Figure~\ref{fig:DSv_exn=20}, where $1 \le a \le 4$. Here there are $6$ non-isomorphic extremal graphs ($a=1$ gives twice the same graph).
	
	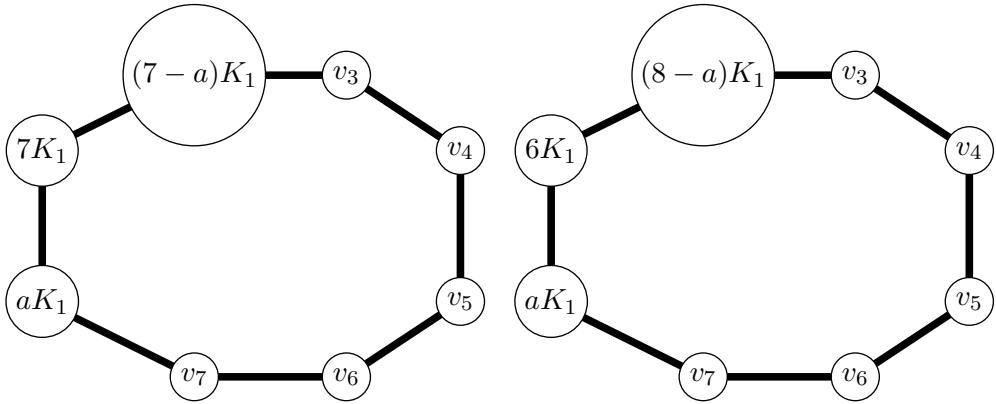
\begin{figure}[h]
		\centering
		\begin{tikzpicture}
		\definecolor{cv0}{rgb}{0.0,0.0,0.0}
		\definecolor{c}{rgb}{1.0,1.0,1.0}
		
		\Vertex[L=\hbox{$aK_1$},x=1cm,y=-1cm]{v0}	
		\Vertex[L=\hbox{$7K_1$},x=1,y=1]{v1}
		\Vertex[L=\hbox{$(7-a)K_1$},x=3,y=2]{v2}
		\Vertex[L=\hbox{$v_3$},x=5,y=2]{v3}
		\Vertex[L=\hbox{$v_4$},x=6.5,y=1]{v4}
		\Vertex[L=\hbox{$v_5$},x=6.5,y=-1]{v5}
		\Vertex[L=\hbox{$v_7$},x=3,y=-2]{v7}
		\Vertex[L=\hbox{$v_6$},x=5,y=-2]{v6}	
		
		\Edge[lw=0.1cm,style={}](v0)(v1)
		\Edge[lw=0.1cm,style={}](v1)(v2)
		\Edge[lw=0.1cm,style={}](v2)(v3)
		\Edge[lw=0.1cm,style={}](v3)(v4)
		\Edge[lw=0.1cm,style={}](v4)(v5)	
		\Edge[lw=0.1cm,style={}](v5)(v6)	
		\Edge[lw=0.1cm,style={}](v6)(v7)	
		\Edge[lw=0.1cm,style={}](v7)(v0)
		\end{tikzpicture}\quad
		\begin{tikzpicture}
		\definecolor{cv0}{rgb}{0.0,0.0,0.0}
		\definecolor{c}{rgb}{1.0,1.0,1.0}
		
		\Vertex[L=\hbox{$aK_1$},x=1cm,y=-1cm]{v0}	
		\Vertex[L=\hbox{$6K_1$},x=1,y=1]{v1}
		\Vertex[L=\hbox{$(8-a)K_1$},x=3,y=2]{v2}
		\Vertex[L=\hbox{$v_3$},x=5,y=2]{v3}
		\Vertex[L=\hbox{$v_4$},x=6.5,y=1]{v4}
		\Vertex[L=\hbox{$v_5$},x=6.5,y=-1]{v5}
		\Vertex[L=\hbox{$v_7$},x=3,y=-2]{v7}
		\Vertex[L=\hbox{$v_6$},x=5,y=-2]{v6}	
		
		\Edge[lw=0.1cm,style={}](v0)(v1)
		\Edge[lw=0.1cm,style={}](v1)(v2)
		\Edge[lw=0.1cm,style={}](v2)(v3)
		\Edge[lw=0.1cm,style={}](v3)(v4)
		\Edge[lw=0.1cm,style={}](v4)(v5)	
		\Edge[lw=0.1cm,style={}](v5)(v6)	
		\Edge[lw=0.1cm,style={}](v6)(v7)	
		\Edge[lw=0.1cm,style={}](v7)(v0)
		\end{tikzpicture}
		\caption{The bipartite graphs with radius $4$ and order $19$ maximizing size}\label{fig:DSv_exn=20}
	\end{figure}

	Chen, Wu and An~\cite{Chen} had conjectured that the graphs $G_{n,r,s}$ minimize the total distance among all graphs of order $n$ and radius $r.$
	This was proven for $n$ sufficiently large compared with $r$ in~\cite{cambie2019extremal}.
	It is natural to think there is a relation between maximizing the size and minimizing the total distance. More edges may imply smaller distances, in particular, more distances are equal to the minimum of $1.$
	Nevertheless, the correspondence is not exact since the cubical graph $Q_3$ also minimizes the total distance among the graphs of order $8$ and radius $3$, but it does not maximize the size.

	Knowing the result for graphs, it is natural to wonder about digraphs as well.
	In $1973$, Fridman~\cite[Thr.5\&6]{F} determined the maximum size of a digraph with given order and outradius. Also, he characterised the extremal graphs. Nevertheless, the extremal digraphs, are not biconnected (also called strongly connected) in this case and so the total distance would be infinite, which implies one cannot compare the relations at this point.
	
	\begin{thr}[\cite{F}]\label{thr:F}
		Let $D$ be a digraph of order $n$ with $rad^+=r\ge 2$.
		Then $\lvert A(D) \rvert \le n(n-r)+\frac{r^2-r-2}{2}$.
		When $r=2$, equality occurs if and only if every vertex of $D$ has outdegree $n-2$.
		For $r \ge 3$, equality occurs if and only if $D \sim \overline{\Gamma}^\star_{n,r,i,s}$ for certain $1\le i\le r-2$ and $1\le s \le n-r.$
	\end{thr}
	
	We now explain the structure of $\overline{\Gamma}^\star_{n,r,i,s}$.
	Let $\overline{\Gamma}_{r}$ be the sum of a transitive tournament on $r$ vertices and the unique longest path $v_1v_2v_3\ldots v_r$ in its complement (as defined before, but with a shift of indices).
	Add an additional vertex $v_0$ and connect it to $v_1$, the unique vertex with out-eccentricity $r-1$ in $\overline{\Gamma}_{r}$. As such we have a graph $\overline{\Gamma}^\star_{r+1}$.
	For an index $1 \le i \le r-2$, take a blow-up by a $K_{s}$ in $v_i$ and by a $K_t$ in $v_{i+1}$, where $t=n-r-s+1.$ As such, we have an extremal digraph $\overline{\Gamma}^\star_{n,r,i,s}.$
	An example where $i=1$ is depicted in Figure~\ref{fig:digraph_Fridman}. Note that $d(u,v_0)=+ \infty$ for all vertices $u$ different from $v_0$ and thus this digraph is not biconnected.

	\begin{figure}[h]
		\centering	
		\begin{tikzpicture}
		
		\definecolor{cv0}{rgb}{0.0,0.0,0.0}
		\definecolor{c}{rgb}{1.0,1.0,1.0}

		\Vertex[L=\hbox{$v_0$},x=0,y=-1]{v0}
		\Vertex[L=\hbox{$K_{s}$},x=2,y=-1]{v1}
		\Vertex[L=\hbox{$K_{n-r-s+1}$},x=5,y=-1]{v2}
		\Vertex[L=\hbox{$v_3$},x=8,y=-1]{v3}
		\Vertex[L=\hbox{$v_4$},x=10,y=-1]{v4}
		
		\Vertex[L=\hbox{$v_r$},x=14,y=-1]{v5}

		\Edge[lw=0.1cm,style={post, right}](v0)(v1)
		\Edge[lw=0.1cm,style={post, right}](v2)(v1)
		\Edge[lw=0.1cm,style={post, right}](v3)(v2)
		\Edge[lw=0.1cm,style={post, right}](v4)(v3)

		\Edge[lw=0.1cm,style={post, right}](v1)(v2)
		\Edge[lw=0.1cm,style={post, right}](v2)(v3)
		\Edge[lw=0.1cm,style={post, right}](v3)(v4)
		\Edge[lw=0.05cm,style={dotted, right}](v4)(v5)
		
		\Edge[lw=0.1cm,style={post, bend left}](v3)(v1)
		\Edge[lw=0.1cm,style={post, bend left}](v4)(v1)	
		\Edge[lw=0.1cm,style={post, bend right}](v5)(v1)		
		\Edge[lw=0.1cm,style={post, bend left}](v4)(v2)	
		\Edge[lw=0.1cm,style={post, bend right}](v5)(v2)	
		\Edge[lw=0.1cm,style={post, bend right}](v5)(v3)	
		\Edge[lw=0.1cm,style={post, bend right}](v5)(v4)		
		\end{tikzpicture}

		\caption{An extremal digraph $\overline{\Gamma}^\star_{n,r,1,s}$ maximizing the size given outradius $r$ and order $n$}
		\label{fig:digraph_Fridman}
	\end{figure}
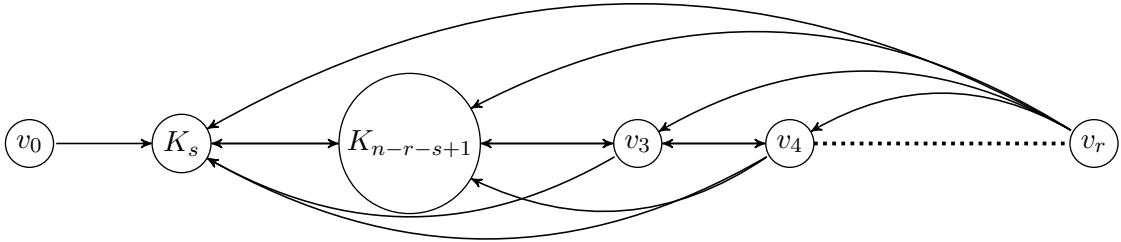

	Dankelmann~\cite{D15} considered the same question in the case of bipartite digraphs and determined the sharp upper bound for the size.
	First note that the bipartite digraph with maximum size is the bidirected $K_{ \lfloor n/2 \rfloor , \lceil n/2 \rceil}$, the digraph containing every arrow between vertices of two balanced independent sets. 
	The extremal digraph with maximum size given outradius, was of the form $\overline{\Gamma}^\star_{n,r,i,s}$ 
	Now one can guess that the extremal digraph(s) given the two conditions, both the outradius and the bipartiteness, is the intersection of the extremal digraphs for the separate conditions. This intuition turns out to be partially true.
	
	For any bipartite digraph with order $n$ and outradius $r$, the size can be upper bounded by the size of the intersection of the digraph $K_{ \lfloor n/2 \rfloor , \lceil n/2 \rceil}$ (maximizing size given bipartiteness) and some $\overline{\Gamma}^\star_{n,r,i,s }$ (maximizing size given outradius), where $s \in \{ \lfloor \frac{n-r+1}{2} \rfloor, \lceil \frac{n-r+1}{2} \rceil \}.$

	Nonetheless, as was the case in the graph case, there are way more extremal digraphs.
	Again there are blow-ups with three stable sets $aK_1, bK_1$ and $cK_1$ with $\lvert a+c-(b+1) \rvert \le 1.$ The precise statement and proof for this is mentioned in Section~\ref{sec:bipdigraph_maxSize}, see Theorem~\ref{thr:bipdigraph_maxSize}.

	Generally, in most problems in graph theory, one studies connected graphs and similarly, biconnected digraphs are interesting, as e.g. only then the notion of total distance makes sense.
	In this paper, we investigate the problem of the maximum size for biconnected digraphs given (out)radius. This question was asked before by Dankelmann~\cite[Prob. 2]{D15}.
	\begin{prob}
		Determine the maximum size of a strong (biconnected) digraph of given order and (out)radius.
	\end{prob}
	
	The digraph version of Vizing's result for biconnected digraphs is trivial when $r \in \{1,2\}.$
	When $r \ge 3$ and $n\ge 2r$, we conjecture that the extremal digraphs are exactly those of the form $D_{n,r,s}$ (remember Figure~\ref{fig:Dnrs}) where $1 \le s \le \frac{n-2r+2}{2}.$ 
	
	\begin{conj}\label{VZdi}
		Let $r \ge 3$ and $n \ge 2r$. Then the maximum size of biconnected digraphs with order $n$ and outradius $r$ is attained by $D_{n,r,1}.$ Furthermore the extremal digraphs are exactly the ones of the form $ D_{n,r,s}$ where $1 \le s \le \frac{n-2r+2}{2}$.
	\end{conj}
	In Section~\ref{sec:Vizing_digraph} we will prove Conjecture~\ref{VZdi} for $r=3$, as well as the case where $r \ge 4$ and $n$ is sufficiently large compared with $r$. By doing this, we make serious progress on \cite[Problem 2]{D15}.
	So at least for $n$ large enough, we conclude the extremal digraphs here correspond with those minimizing the total distance under the same conditions.

	When one is working with the diameter instead of radius, the results are known and one can see similar relationships, when comparing the extremal (di)graphs attaining the minimum average distance (Plesn\'{i}k~\cite{P84}) and maximum size (Ore~\cite{Ore68}, the digraph case being folklore, as mentioned in e.g.~\cite{D15}). Here the set of (di)graphs attaining the minimum average distance is a subset of the set of (di)graphs attaining the maximum size in this case.
	The extremal (di)graphs in the latter case are formed by
	taking two blow-ups at $2$ consecutive non-end vertices of a path of length $d$ (graph case) or at the digraph $\overline{ \Gamma}_{d+1}$ which is the sum of a transitive tournament on $d+1$ vertices and the unique longest path in its complement (digraph case).
	To get the extremal (di)graphs in the former case one needs to take blow-ups only at the 2 central vertices (or one central vertex).

	\begin{figure}[h]
		\begin{center}

			\begin{tikzpicture}
			\definecolor{cv0}{rgb}{0.0,0.0,0.0}
			\definecolor{c}{rgb}{1.0,1.0,1.0}
			
			\Vertex[L=\hbox{$v_0$},x=1cm,y=0.0cm]{v0}	
			\Vertex[L=\hbox{$v_1$},x=3,y=0]{v1}
			\Vertex[L=\hbox{$v_2$},x=5,y=0]{v2}
			\Vertex[L=\hbox{$v_3$},x=7,y=0]{v3}
			\Vertex[L=\hbox{$v_4$},x=9,y=0]{v4}
			\Vertex[L=\hbox{$v_5$},x=11,y=0]{v5}

			\Edge[lw=0.1cm,style={post, right}](v0)(v1)
			\Edge[lw=0.1cm,style={post, right}](v1)(v2)
			\Edge[lw=0.1cm,style={post, right}](v2)(v3)
			\Edge[lw=0.1cm,style={post, right}](v3)(v4)
			\Edge[lw=0.1cm,style={post, right}](v4)(v5)

			\Edge[lw=0.1cm,style={post, right}](v1)(v0)
			\Edge[lw=0.1cm,style={post, right}](v2)(v1)
			\Edge[lw=0.1cm,style={post, right}](v3)(v2)
			\Edge[lw=0.1cm,style={post, right}](v4)(v3)
			\Edge[lw=0.1cm,style={post, right}](v5)(v4)

			\Edge[lw=0.1cm,style={post, bend left}](v2)(v0)
			\Edge[lw=0.1cm,style={post, bend right}](v3)(v0)
			\Edge[lw=0.1cm,style={post, bend left}](v4)(v0)
			\Edge[lw=0.1cm,style={post, bend right}](v5)(v0)
			\Edge[lw=0.1cm,style={post, bend left}](v3)(v1)
			\Edge[lw=0.1cm,style={post, bend right}](v4)(v1)	
			\Edge[lw=0.1cm,style={post, bend left}](v5)(v1)		
			\Edge[lw=0.1cm,style={post, bend left}](v4)(v2)	
			\Edge[lw=0.1cm,style={post, bend right}](v5)(v2)	
			\Edge[lw=0.1cm,style={post, bend left}](v5)(v3)

			\end{tikzpicture}
		\end{center}
		\caption{The digraph $\overline{ \Gamma}_{d+1}$ for $d=5$}
		\label{fig:Gamd+1}
	\end{figure}
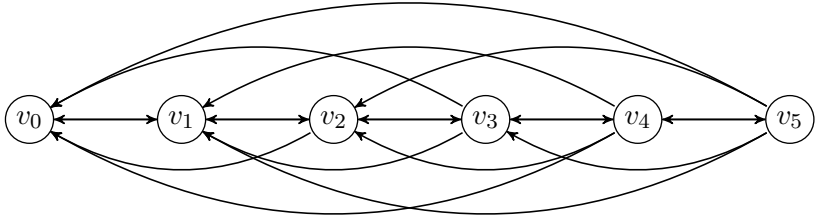

	Nevertheless, the statement that the set of extremal digraphs attaining the minimum average distance is a subset of the set of extremal digraphs maximizing the size is not always true, as is indicated in Subsection~\ref{subsec:maxMgivenRad}. 
	In that subsection we consider the question of maximizing the size given order and radius for digraphs.
	%
	
	Finally, in Section~\ref{sec:bipbicDigraphs_maxSize}, we consider the question on the maximum size for bipartite biconnected digraphs again. In contrast with the results of Theorem~\ref{thr:DSv11} and Theorem~\ref{thr:bipdigraph_maxSize}, we conjecture that the extremal biconnected bipartite digraph of maximum size given order and outradius $r \ge 4$ is unique, being a bipartite subdigraph of a balanced $D_{n,r,s}.$
	The precise statement is mentioned in Conjecture~\ref{conj:extr_bicon_bip_digraph_rad+ge4}.
	Furthermore in Section~\ref{sec:bipbicDigraphs_maxSize} we prove this conjecture for $n$ sufficiently large in terms of $r$ when $r$ is even. 
	So the difference with the two other theorems on the maximum size of bipartite graphs and digraphs is clear for these values.
	As such, we formulated a conjecture for \cite[Prob. 3]{D15} and proved it asymptotically (for even $r$).
	
	\subsection{Structure of the paper and outline of proofs}
	
	In this subsection, we give an overview with the main ideas used in Sections~\ref{sec:Vizing_digraph},~\ref{sec:bipdigraph_maxSize} and~\ref{sec:bipbicDigraphs_maxSize}.
	
	In Section~\ref{sec:Vizing_digraph}, we start with Theorem~\ref{thr:Vz_rad+3}. Here we prove Conjecture~\ref{VZdi} for $r=3.$ The main idea here is that by definition of the outradius, the outdegree of every vertex is bounded by $n-3$, while there are some high degree vertices for which equality is attained.
	For these vertices, there are precisely two other vertices which are at distance $2$ and $3$. Hence we know some structure of the extremal digraphs and we conclude after performing some case analysis. 
	
	We continue with progress on Conjecture~\ref{VZdi}, by showing that the conjecture does hold for $r>3$ and $n$ sufficiently large.
	For this, we go along the lines of the proof strategy sketched in the third section of~\cite{SC19}. First, in Lemma~\ref{lem2di_size} we prove that for large values of $n$, there does exist a vertex in the extremal digraph that does not influence the distance measures. This has been done by counting arguments applied to a big clique contained in the extremal digraph.
	Next in Proposition~\ref{proplem_size} we show that the total degree of any vertex in a digraph with order $n$ and outradius $r$ is bounded by $2(n-1)-(2r-3)$ and that equality can only occur under certain conditions.
	Here we look for the minimum total degree in the complement of the digraph, by taking the eccentricity conditions (consequence of the given outradius) into consideration.
	Finally in Theorem~\ref{mainVz_bicDi} we wrap up the proof, by noticing that for sufficiently large $n$, there exists a vertex $v$ in a subdigraph of any extremal digraph for which equality is attained in Proposition\ref{proplem_size}. From the resulting information that can be derived, after some case analysis we conclude.
	
	In Subsection~\ref{subsec:maxMgivenRad} we add some thoughts on the same question when the radius of a digraph is involved instead of the outradius. 
	
	In Section~\ref{sec:bipdigraph_maxSize}, we give the precise characterization of the bipartite digraphs with maximum size given their outradius. The maximum size was determined by Dankelmann~\cite{D15}. We extend his ideas. Hereby we determine upperbounds for the outneighbourhood of every vertex in terms of the number of vertices in the outneighbourhoods $N_i$ at distance $i$ from a vertex $v$ with $\ecc^+(v)=r$.
	Next, we prove an inequality on integers using a substitution and we are careful about the conditions for which equality does hold.
	
	Finally, in Section~\ref{sec:bipbicDigraphs_maxSize} we state Conjecture~\ref{conj:extr_bicon_bip_digraph_rad+ge4} about the bipartite biconnected digraphs of order $n$ and outradius $r \ge 4$ attaining the maximum size. 
	We prove it in the case that $r$ is even and $n$ is sufficiently large in terms of $r.$ Hereby we again start by proving that certain substructures are present in an extremal digraph, for example a large bidirected complete bipartite digraph all of whose vertices have large degree.
	Then by a counting argument, we prove that there do exist vertices in the extremal digraph(s) that can be removed without changing the remaining distances, once the order is sufficiently large.
	By noticing that these vertices can only attain a certain upper bound on their total degree if some restrictions are met, we are ready to do some technical case analysis to finish the proof of Theorem~\ref{thr:bipbicDigraphMaxSize}.

	\section{Notation and definitions}\label{not&def}
	
	A graph will be denoted by $G=(V,E)$ and 
	a digraph will be denoted by $D=(V,A).$
	The order $\lvert V \rvert$ will be denoted by $n$. 
	A clique or bidirected clique on $n$ vertices will be denoted by $K_n$. 
	A cycle or directed cycle of length $k$ will be denoted by $C_k$.
	The clique number of a graph $G$, $\omega(G)$, is the order of the largest clique which is a subgraph of $G.$
	The complement $G^c$ of graph $G=(V,E)$ is the graph with vertex set $V$ and edge set $E^c=\binom{V}{2} \backslash E.$ 
	The complement $D^c$ of a digraph $D$ is defined similarly, where the set of directed edges is the complement with respect to the edges of a bidirected clique.
	The reverse of a digraph $D=(V,A)$ is the digraph $D'=(V,A')$ with $A'=\{\vc{xy}\mid \vc{yx} \in A\}.$
	The intersection of two graphs $G=(V,E)$ and $G'=(V',E')$ is the graph $G \cap G'=(V \cap V', E \cap E')$ and analogously for digraphs.
	
	The degree of a vertex in a graph $\deg(v)$ equals the number of neighbours of the vertex $v$, i.e. $\deg(v)=\lvert N(v) \rvert.$
	In a digraph, we denote with $N^-(v)$ and $N^+(v)$ the (open) in- and outneighbourhood of a vertex $v$.
	The indegree $\deg^-$ and outdegree $\deg^+$ of a vertex $v$, equals the number of arrows ending in or starting from the vertex $v$, i.e. $\deg^+(v)=\lvert N^+(v) \rvert $ and $\deg^-(v)=\lvert N^-(v) \rvert $.
	The total degree $\deg$ of a vertex $v$ in a digraph is the sum of the in- and outdegree, i.e. $\deg(v)=\deg^+(v)+\deg^-(v).$
	
	Let $d(u, v)$ denote the distance between vertices $u$ and $v$ in a graph $G$ or digraph $D$, i.e. the number of edges or arrows in a shortest path from $u$ to $v$. 
	The eccentricity of a vertex $v$ in a graph equals $\ecc(v)=d(v,V)=\max_{u \in V} d(v,u).$
	The radius and diameter of a graph on vertex set $V$ are respectively equal to $\min_{v \in V} \ecc(v)$ and $\max_{v \in V} \ecc(v)=\max_{u,v \in V} d(u,v).$

	In the case of digraphs, the distance function between vertices is not symmetric and so there is a difference between the inner- and outer eccentricity 
	$\ecc^- (v)=d(V,v)=\max_{u \in V} d(u,v)$ and $\ecc^+(v)=d(v,V)=\max_{u \in V} d(v,u).$
	We use the conventions as in e.g. \cite{JG}.
	Radius, inradius and outradius are defined in Subsection $2.1$ in \cite{JG} or Subsection $3.1$ in \cite{JG2} but for clarity we define them in the next sentences.
	The in- and outradius of a digraph $D$ are defined by 
	$\rad^{-}(D)= \min\{d(V,x)\mid x \in V\}$ and $\rad^{+}(D)= \min\{d(x,V)\mid x \in V\}$.
	The radius of a digraph $D$ is defined as $\rad(D)=\min\{ \frac{d(x,V)+ d(V,x)}2 \mid x \in V \}.$ Sometimes authors refer to the outradius as radius (see e.g. Subsection $1.4$ in~\cite{CLZ}), as outradius is the most common one between those three definitions.
	
	The total distance, also called the Wiener index, of a graph $G$ equals the sum of distances between all unordered pairs of vertices, i.e. $W(G)=\sum_{\{u,v\} \subset V} d(u,v).$ 
	The average distance of a graph is $\mu(G)=\frac{W(G)}{\binom{n}{2}}$. 
	The Wiener index of a digraph equals the sum of distances between all ordered pairs of vertices, i.e. $W(D)=\sum_{(u,v) \in V^2} d(u,v).$
	The average distance of the digraph is $\mu(D)=\frac{W(D)}{n^2-n}.$
	A digraph is called biconnected (or a strong digraph) if $d(u,v)$ is finite for any $2$ vertices $u$ and $v$. 
	
	The statement $f(x)=O(g(x))$ as $x \to \infty$ implies that there exist fixed constants $x_0, M>0$, such that for all $x \ge x_0$ we have $\lvert f(x) \rvert \le M \lvert g(x) \rvert .$
	Analogously, $f(x)=\Omega(g(x))$ as $x \to \infty$ implies that there exist fixed constants $x_0, M>0$, such that for all $x \ge x_0$ we have $\lvert f(x) \rvert \ge M \lvert g(x) \rvert.$
	If $f(x)=\Omega(g(x))$ and $f(x)=O(g(x))$ as $x \to \infty$, then one uses $f(x)=\Theta(g(x))$ as $x \to \infty$. Sometimes we do not write the "as $x \to \infty$" if the context is clear.
	
	\begin{defi}
		Given a graph $G$ and a vertex $v$, the blow-up of a vertex $v$ of a graph $G$ by a graph $H$ is constructed as follows.
		Take $G \backslash v$ and connect all initial neighbours of $v$ with all vertices of a copy of $H.$	
		When taking the blow-up of a vertex $v$ of a digraph $D$ by a digraph $H$, a directed edge between a vertex $w$ of $D \backslash v$ and a vertex $z$ of $H$ is drawn if and only if initially there was a directed edge between $w$ and $v$ in the same direction. 
		When taking the blow-ups of multiple vertices, for neighbouring vertices $v_1$ and $v_2$, all vertices of the corresponding graphs $H_1$ and $H_2$ are connected as well in the blow-up (possible in one direction in the digraph case). Equivalently, one can take the blow-up of the different vertices one at a time at the resulting graph of the blow-up in the previous step. 
	\end{defi}
	
	Let $\overline{\Gamma}_{d+1}$ be the sum of a transitive tournament on $d+1$ vertices and the unique longest path in its complement. Equivalently, let $V=\{v_0,v_1, \ldots, v_d\}$ be the set of its vertices and $A=\{\vc{v_i v_j} \mid i \ge j-1 \}.$ An example for $d=5$ is presented in Figure~\ref{fig:Gamd+1}. 
	
	Let $\overline{\Gamma}_{n,d,i,s}$ (see Figure~\ref{fig:Gamnd1s}) be the digraph obtained from $\overline{\Gamma}_{d+1}$ by taking the blow-up of $v_i$ by a bidirected clique $K_s$ and a blow-up of $v_{i+1}$ by a bidirected clique $K_t=K_{n-d+1-s}$.

	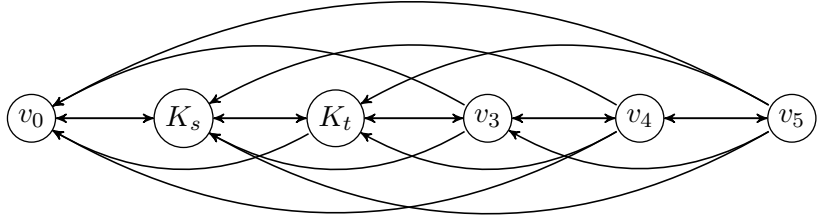
\begin{figure}[h]
		\centering

		\begin{tikzpicture}
		\definecolor{cv0}{rgb}{0.0,0.0,0.0}
		\definecolor{c}{rgb}{1.0,1.0,1.0}
		
		\Vertex[L=\hbox{$v_0$},x=1cm,y=0.0cm]{v0}	
		\Vertex[L=\hbox{$K_s$},x=3,y=0]{v1}
		\Vertex[L=\hbox{$K_t$},x=5,y=0]{v2}
		\Vertex[L=\hbox{$v_3$},x=7,y=0]{v3}
		\Vertex[L=\hbox{$v_4$},x=9,y=0]{v4}
		\Vertex[L=\hbox{$v_5$},x=11,y=0]{v5}

		\Edge[lw=0.1cm,style={post, right}](v0)(v1)
		\Edge[lw=0.1cm,style={post, right}](v1)(v2)
		\Edge[lw=0.1cm,style={post, right}](v2)(v3)
		\Edge[lw=0.1cm,style={post, right}](v3)(v4)
		\Edge[lw=0.1cm,style={post, right}](v4)(v5)

		\Edge[lw=0.1cm,style={post, right}](v1)(v0)
		\Edge[lw=0.1cm,style={post, right}](v2)(v1)
		\Edge[lw=0.1cm,style={post, right}](v3)(v2)
		\Edge[lw=0.1cm,style={post, right}](v4)(v3)
		\Edge[lw=0.1cm,style={post, right}](v5)(v4)

		\Edge[lw=0.1cm,style={post, bend left}](v2)(v0)
		\Edge[lw=0.1cm,style={post, bend right}](v3)(v0)
		\Edge[lw=0.1cm,style={post, bend left}](v4)(v0)
		\Edge[lw=0.1cm,style={post, bend right}](v5)(v0)
		\Edge[lw=0.1cm,style={post, bend left}](v3)(v1)
		\Edge[lw=0.1cm,style={post, bend right}](v4)(v1)	
		\Edge[lw=0.1cm,style={post, bend left}](v5)(v1)		
		\Edge[lw=0.1cm,style={post, bend left}](v4)(v2)	
		\Edge[lw=0.1cm,style={post, bend right}](v5)(v2)	
		\Edge[lw=0.1cm,style={post, bend left}](v5)(v3)

		\end{tikzpicture}

		\caption{The digraph $\overline{ \Gamma}_{n,d,1,s}$ for $d=5$}
		\label{fig:Gamnd1s}
	\end{figure}
	
	Let $G_{n,r,s}$, where $n \ge 2r$ and $1 \le s \le \frac{n-2r+2}{2}$, be the graph obtained by taking two blow-ups of two consecutive vertices in a cycle $C_{2r}$ by cliques $K_s$ and $K_{n-2r+2-s}$ respectively. 
	Note that $\omega(G_{n,r,s})=n-2r+2.$
	
	
	Let $D_{2r,r,1}$ be a digraph with $V=\{v_1,v_2, \ldots v_r\} \cup \{w_1, w_2, \ldots, w_r\}$ and $$A=\{\vc{v_i v_j} \mid i \ge j-1 \}\cup \{\vc{w_i w_j} \mid i \ge j-1 \} \cup \{\vc{v_i w_1} \mid 1 \le i \le r \} \cup \{\vc{w_i v_1} \mid 1 \le i \le r \}.$$

	Let $D_{n,r,s}$, $n \ge 2r$ and $1 \le s \le \frac{n-2r+2}{2}$, be the digraph obtained by taking the blow-up of $v_1$ by a bidirected clique $K_s$ and a blow-up of $w_1$ by a bidirected clique $K_{n-2r+2-s}$. A depiction of this digraph is given in Figure~\ref{fig:Dnrs}.

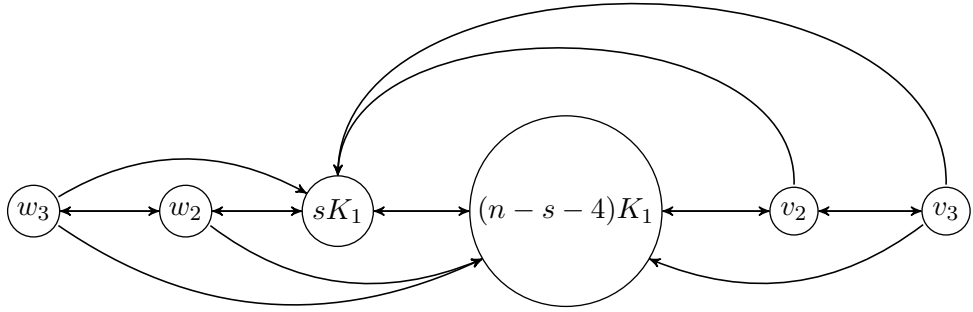
\begin{figure}[h]
	\centering

	\begin{tikzpicture}
	\definecolor{cv0}{rgb}{0.0,0.0,0.0}
	\definecolor{c}{rgb}{1.0,1.0,1.0}

	\Vertex[L=\hbox{$v_2$},x=13,y=0]{v1}
	\Vertex[L=\hbox{$v_3$},x=15,y=0]{v2}
	
	\Vertex[L=\hbox{$w_2$},x=5,y=0]{w1}
	\Vertex[L=\hbox{$w_3$},x=3,y=0]{w2}
	\Vertex[L=\hbox{$sK_1$},x=7,y=0]{w0}
	
	\Vertex[L=\hbox{$(n-s-4)K_1$},x=10cm,y=0.0cm]{v0}
	
	\Edge[lw=0.1cm,style={post, right}](v0)(v1)
	\Edge[lw=0.1cm,style={post, right}](v1)(v2)
	\Edge[lw=0.1cm,style={post, right}](w0)(v0)
	\Edge[lw=0.1cm,style={post, right}](v1)(v0)
	\Edge[lw=0.1cm,style={post, right}](v2)(v1)
	\Edge[lw=0.1cm,style={post, right}](v0)(w0)
	\Edge[lw=0.1cm,style={post, right}](w1)(w2)
	\Edge[lw=0.1cm,style={post, right}](w0)(w1)
	\Edge[lw=0.1cm,style={post, right}](w1)(w0)
	\Edge[lw=0.1cm,style={post, right}](w2)(w1)

	\Edge[lw=0.1cm,style={post, bend left}](w2)(w0)
	\Edge[lw=0.1cm,style={post, bend right}](w2)(v0)
	\Edge[lw=0.1cm,style={post, bend right}](w1)(v0)
	
	\Edge[lw=0.1cm,style={post, bend right=90}](v2)(w0)
	\Edge[lw=0.1cm,style={post, bend left}](v2)(v0)
	\Edge[lw=0.1cm,style={post, bend right=90}](v1)(w0)
	
	\end{tikzpicture}

	\caption{The digraph $D_{n,r,s}$ for $r=3$}
	\label{fig:Dnrs}
\end{figure}

\section{Maximum size biconnected digraphs}\label{sec:Vizing_digraph}

In this section, we prove that Conjecture~\ref{VZdi} is true when $r=3$, as well as in the case $r>3$ and $n$ large enough with respect to $r$.

\begin{thr}\label{thr:Vz_rad+3}
	Let $n \ge 6.$ Then any biconnected digraph $D=(V,A)$ with order $n$ and outradius $r=3$ satisfies $\lvert A \rvert \le (n-2)^2$. Equality holds if and only if $ D \cong D_{n,r,s}$ for some $1 \le s \le \frac{n-2r+2}{2}$.
\end{thr}

\begin{proof}
	Note that $\lvert A \rvert = \sum_{v \in V} \deg^+(v).$
	Since the outradius of $D$ is $3$, we know $\ecc^+(v) \ge 3$ for every $v \in V$, which implies that $\deg^+(v) \le n-3$ for every $v \in V$.
	If $\lvert A (D) \rvert \ge n^2-4n+	4$, we know there are at least $4$ vertices with outdegree equal to $n-3.$
	For any vertex $v$, let $F(v)$ be the set of vertices $x$ in $V$ different from $v$ for which there is no arrow from $v$ to $x$, i.e. $F(v)=V \backslash N^+[v].$ Note that $\deg^+(v)=n-1-\lvert F(v)\rvert.$
	Let $c$ (being a centre of the digraph) have outdegree $n-3$ and let $a_2$ and $a_3$ be the $2$ vertices such that $d(c,a_2)=2$ and $d(c,a_3)=3.$
	Let $Y$ be the set of vertices $y$ such that $d(c,y)=1=d(y,a_2)$ and $X$ be the remaining vertices. Note that $X$ is not empty, as otherwise, $D$ has at most the size of a digraph formed by taking a blow-up of a vertex of a directed $C_4$, which has a smaller size than $(n-2)^2.$ For this, remark that there cannot be a directed edge from some $y \in Y$ or $a_2$ to $c$ as then the outeccentricity of that vertex is at most $2$. So there is a directed edge from $a_3$ to $c$, i.e. 
	$\vc{a_3c} \in A(D)$. Now for the same reason, there cannot be directed edges from $a_3$ to $a_2$ or some $y \in Y$.
	It is possible there is an edge from $a_2$ to some $y \in Y$, but then these vertices cannot have $Y \subset N^+[y]$ and hence the outdegree of these vertices is lower than in the blow-up of the directed $C_4,$ from which the bound on the size follows.
	\begin{figure}[h]
		\centering	
		\begin{tikzpicture}
		
		\definecolor{cv0}{rgb}{0.0,0.0,0.0}
		\definecolor{c}{rgb}{1.0,1.0,1.0}

		\Vertex[L=\hbox{$X$},x=0,y=-1]{x}
		\Vertex[L=\hbox{$c$},x=2,y=-1]{c}
		\Vertex[L=\hbox{$Y$},x=4,y=-1]{y}
		\Vertex[L=\hbox{$a_2$},x=6,y=-1]{a2}
		\Vertex[L=\hbox{$a_3$},x=8,y=-1]{a3}

		\Edge[lw=0.1cm,style={post, right}](c)(x)
		\Edge[lw=0.1cm,style={post, right}](c)(y)
		\Edge[lw=0.1cm,style={post, right}](y)(a2)
		\Edge[lw=0.1cm,style={post, right}](a2)(a3)

		\Edge[lw=0.1cm,style={post, bend right}, color=red](y)(c)
		\Edge[lw=0.1cm,style={post, bend left}, color =red](x)(a2)	
		\Edge[lw=0.1cm,style={post, bend right}, color =red](x)(a3)
		\Edge[lw=0.1cm,style={post, bend right}, color =red](c)(a3)
		\Edge[lw=0.1cm,style={post, bend right}, color =red](c)(a2)
		\Edge[lw=0.1cm,style={post, bend left}, color =red](y)(a3)
		\end{tikzpicture}

		\caption{Partial presentation of an extremal digraph with $\rad^+=3$ maximizing the size}
		\label{fig:towards_extremal_digraph_rad+3}
	\end{figure}
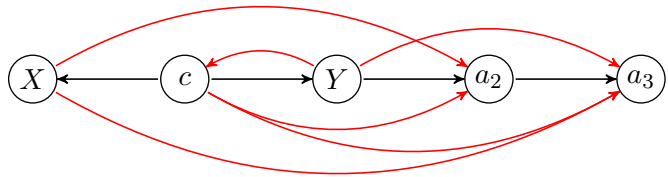
	
	Note that there is no directed edge from $y$ to $c$ since otherwise $\ecc^+(y)=2$.
	Part of the arrows for an extremal digraph have been presented (black) in Figure~\ref{fig:towards_extremal_digraph_rad+3}, as well as forbidden arrows (red).
	
	As there were at least $4$ vertices for which the outdegree equals $n-3$, there is at least one such a vertex in $X \cup Y$.
	Now we do some case analysis.

	\textbf{Case 1} There is a $y \in Y$ with $\deg^+(y)=n-3$.
	Hence there is an arrow from $y$ to all vertices different from $c$ and $a_3,$ i.e. $F(y)= \{a_3,c\}.$
	Now we see $\{a_2,a_3,c\} \subset F(x)$ for all $x \in X$, due to respectively the definition of $X$, $d(c,a_3)=3$ and $d(c,x)=1$, and the conditions $d(y,c)=3$ and $d(y,x)=1$.
	Similarly one has $\{c,y\} \subset F(a_2)$, $Y \cup \{a_2\} \subset F(a_3)$ and $\{c,a_3\} \subset F(y')$ for all $y' \in Y$ different from $y$. 
	Using these observations, we have $\lvert A (D) \rvert = n^2-n- \sum_{v \in V} \lvert F(v) \rvert \le n^2-4n+	4$ and equality is not possible, since then $\ecc^+(a_2)=2.$
	
	\textbf{Case 2} The set $X_2$ containing the vertices $x \in X$ with $\deg^+(x)=n-3$ has size at least $2$, note that we have $F(x)=\{a_2,a_3\}$ for those $x$.
	Furthermore we remark that $X_2 \cup \{c, a_3\} \subset F(y)$ for every $y \in Y$ (otherwise $\ecc^+(y)<3$) and $ X_2 \cup \{c\} \subset F(a_2)$ (otherwise $\ecc^+(a_2)<3$).
	We get $\lvert A (D) \rvert = n^2-n- \sum_{v \in V} \lvert F(v) \rvert \le n^2-4n+	4- \lvert Y \rvert \left( \lvert X_2 \rvert -1 \right) <n^2-4n+	4$ in this case.
	
	\textbf{Case 3} In the last case, there is one vertex $x_2 \in X$ with $\deg^+(x_2)=n-3$.
	Furthermore we need $\deg^+(a_2)=\deg^+(a_3)=\deg^+(c)=n-3$ and $\deg^+(v)=n-4$ for all the remaining vertices $v$.
	We know that $F(x_2)=F(c)=\{a_2, a_3\}$ (by the definitions) and $F(a_2)=\{c, x_2\}$ (as otherwise $\ecc^+(a_2)<3$).
	Also for every $y \in Y$ it should be that $F(y)=\{a_3,c,x_2\}$ as otherwise $\ecc^+(y) \ge 3$. 
	If $X=\{x_2\},$ then we get a contradiction as there should be a directed edge from $a_3$ to at least one of $c$ and $x_2$ (since $D$ is biconnected) and consequently there cannot be a directed edge to some $y \in Y$ or to $a_2$ (since $\ecc^+(a_3) \ge 3$). But then $\lvert F(a_3) \rvert \ge |Y|+1>2,$ contradiction.
	
	So now assume $X\not=\{x_2\},$ i.e. also the set $X \backslash x_2$ in Figure~\ref{fig:towards_extremal_digraph_rad+3_case3} is non-empty.
	By definition we already have $\{a_2,a_3\} \subset F(x)$ for any $x \in X.$
	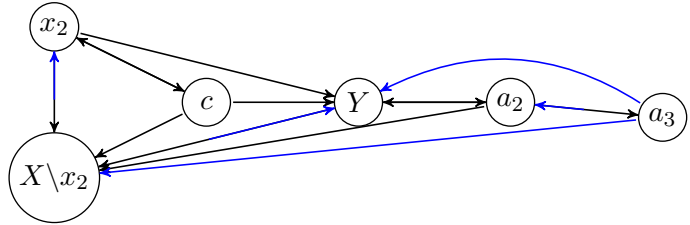
\begin{figure}[h]
		\centering	
		\begin{tikzpicture}
		
		\definecolor{cv0}{rgb}{0.0,0.0,0.0}
		\definecolor{c}{rgb}{1.0,1.0,1.0}

		\Vertex[L=\hbox{$x_2$},x=0,y=0]{x2}
		\Vertex[L=\hbox{$X\backslash x_2$},x=0,y=-2]{x}
		\Vertex[L=\hbox{$c$},x=2,y=-1]{c}
		\Vertex[L=\hbox{$Y$},x=4,y=-1]{y}
		\Vertex[L=\hbox{$a_2$},x=6,y=-1]{a2}
		\Vertex[L=\hbox{$a_3$},x=8,y=-1.2]{a3}

		\Edge[lw=0.1cm,style={post, right}](c)(x2)
		\Edge[lw=0.1cm,style={post, right}](c)(x)
		\Edge[lw=0.1cm,style={post, right}](c)(y)
		\Edge[lw=0.1cm,style={post, right}](y)(a2)
		\Edge[lw=0.1cm,style={post, right}](a2)(a3)
		\Edge[lw=0.1cm,style={post, right}](a2)(x)
		\Edge[lw=0.1cm,style={post, right}](a2)(y)
		\Edge[lw=0.1cm,style={post, right}](x2)(c)
		\Edge[lw=0.1cm,style={post, right}](x2)(y)
		\Edge[lw=0.1cm,style={post, right}](x2)(x)
		\Edge[lw=0.1cm,style={post, right}](y)(x)
		
		\Edge[lw=0.1cm,style={post, right}, color=blue](2,-1.5)(y)
		\Edge[lw=0.1cm,style={post, right}, color =blue](0,-1)(x2)
		\Edge[lw=0.1cm,style={post, right}, color =blue](7,-1.1)(a2)

		\Edge[lw=0.1cm,style={post, right}, color =blue](a3)(x)	
		\Edge[lw=0.1cm,style={post, bend right}, color=blue](a3)(y)		
		\end{tikzpicture}

		\caption{Partial presentation of an extremal digraph in case $3$}
		\label{fig:towards_extremal_digraph_rad+3_case3}
	\end{figure}
	
	Since $\ecc^+(a_2)=3$ and $F(a_2)=\{c, x_2\},$ we see that there can't be directed edges from $X \backslash \{x_2\}$ to both $c$ and $x_2$.
	Assume without loss of generality that there is no directed edge towards $c$, i.e. $F(x)=\{a_2,a_3,c\}$ for every $x \in X \backslash \{x_2\}$.
	Finally, we observe that there is no directed edge from $a_3$ to $c$ since $d(a_2,c)=3.$ If $\vc{a_3x_2} \in A$, we have $\ecc^+(a_3)<3$ since there is a directed edge from $a_3$ to $Y$ or $a_2$. This would implies that $F(a_3))=\{c,x_2\}$
	%
	and we conclude $D$ is isomorphic to $D_{n,r,s}$ where $s=\min \{ |X|-1, |Y|\}$ (see Figure~\ref{fig:towards_extremal_digraph_rad+3_case3} taking the blue edges also into account).
\end{proof}

Now, we prove that Conjecture~\ref{VZdi} holds when $n$ is large enough for a fixed $r$.
First, we note that the analogue of \cite[Lemma~4.3]{cambie2019extremal} holds.

\begin{lem}\label{lem2di_size}
	Let $r \ge 3$.
	There is a value $n_0(r)$ such that for any $n \ge n_0$ and any digraph $D=(V,A)$ of order $n$ and outradius $r$ with maximum size among such digraphs, there is a vertex $v \in D$ such that $D \backslash v$ has outradius $r$ and the distance between any $2$ vertices of $D \backslash v$ equals the distance between them in $D.$
\end{lem}

\begin{proof}
	Note that such a digraph $D$ satisfies $\lvert A \rvert > n(n-1-t)$, where $t:=t(r)=2r-3$ due to the example $D_{n,r,s}.$
	We know from \cite[Lemma~4.1]{cambie2019extremal} that $D$ contains a clique $K_k$ with $k \ge \frac{n}{8t}$, such that the (total) degree of all its vertices is at least $2(n-1)-4t.$
	Let $n_0:=n_0(r)= 8t(40t^2+4t+1)+1$.
	Since $k \ge \frac{n}{8t} >40t^2+4t+1 \ge 32t^2+4t+1+\frac{tn}{k}$ when $n\ge n_0$, the result follows from \cite[Lemma~4.2]{cambie2019extremal}.
\end{proof}

\begin{prop}\label{proplem_size}
	Let $D=(V,A)$ be a digraph of order $n$ and outradius $r$.
	Then for any vertex $v$, the total degree $\deg(v)\le 2(n-1)-(2r-3).$
	Equality can occur if and only if
	\begin{itemize}
		\item $\deg^-(v)=n-1$ and $\deg^+(v)=n-1-(2r-3)$,
		\item there exists two disjoint directed paths $vu_1u_2\ldots u_{r}$ and $v w_2 w_3 \ldots w_r$ in $D$ with $d(v,u_r)=r$ and $d(v,w_r)=r-1$	
	\end{itemize}
\end{prop}

\begin{proof}
	Let $\ecc^+(v)=r' \ge r$ and assume $vu_1u_2\ldots u_{r'}$ is a directed path in $D$ with $d(v,u_{r'}) =r'.$
	Let $i$ be the smallest index with $r'-r+1 \le i \le r'$ for which there is an arrow from $u_i$ to $v.$ Note that this one has to exist, or otherwise $\deg^-(v)\le(n-1)-r$ and $\deg^+(v)\le (n-1)-(r-1)$ and so $\deg(v) \le 2(n-1)-(2r-1).$\\
	We know there exists a vertex $x$ with $d(u_{r'-r+1},v)+d(v,x)\ge d(u_{r'-r+1},x) \ge r$ by the triangle inequality and $\ecc^+(u_{r'-r+1})\ge r$. 
	To have some graphical presentation, Figure~\ref{fig:partialconfiguration} presents a possible scenario.
	Next to that, we have $d(u_{r'-r+1},v)\le d(u_{r'-r+1},u_i)+1=i+r-r'.$
	These $2$ observations together imply that $d(v,x) \ge r-(i-r'+r)=r'-i.$\\
	The shortest path from $v$ to $x$ does not pass any $u_j$ for some $r' \ge j \ge r'-r+1$ by the choice of $d(v,u_{r'})$ being maximal and $d(u_{r'-r+1},x)\ge r >d(u_{r'-r+1},u_{r'}).$
	So we have at least $r'-i-1$ vertices $w$, on the shortest path from $v$ to $x$, for which there is no edge from $v$ to $w$.\\
	There is no edge from $v$ to any of the $r-1$ $u_j$ when $r'-r+2 \le j \le r'$ either.
	Furthermore $v$ is not in the outneighbourhood of any of the $i+r-r'-1$ $v_j$ with $r'-r+1 \le j < i$. 	If $r'>r$, then $d(v,u_{r'-r+1})>1$ as well and in this case $\deg(v) \le 2(n-1)-(2r-2)$ because we have found already $(r'-i-1)+(r-1)+(i+r-r'-1)+1=2r-2$ edges which are missing.
	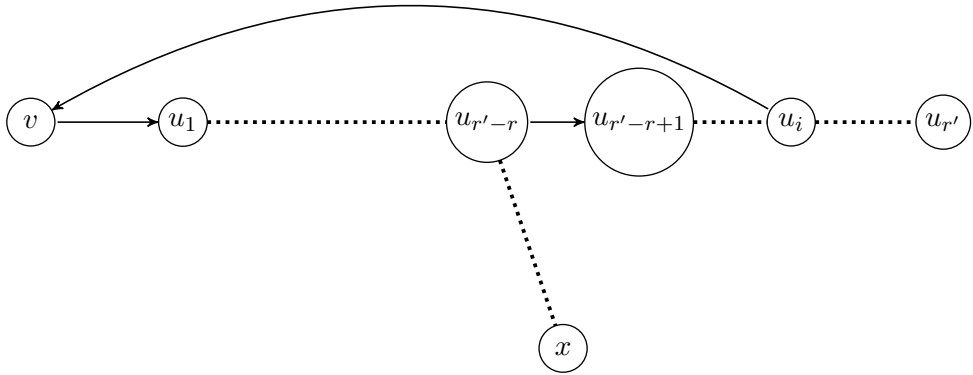
\begin{figure}[h]
		\centering

		\begin{tikzpicture}
		\definecolor{cv0}{rgb}{0.0,0.0,0.0}
		\definecolor{c}{rgb}{1.0,1.0,1.0}
		
		\Vertex[L=\hbox{$v$},x=0cm,y=0cm]{v0}	
		\Vertex[L=\hbox{$u_1$},x=2,y=0]{v1}
		
		\Vertex[L=\hbox{$u_{r'-r}$},x=6,y=0]{v5}
		\Vertex[L=\hbox{$u_{r'-r+1}$},x=8,y=0]{v2}
		\Vertex[L=\hbox{$u_i$},x=10,y=0]{v3}
		\Vertex[L=\hbox{$u_{r'}$},x=12,y=0]{v4}
		\Vertex[L=\hbox{$x$},x=7,y=-3]{x}
		
		\Edge[lw=0.1cm,style={post, right}](v0)(v1)
		\Edge[lw=0.1cm,style={post, right}](v5)(v2)
		\Edge[lw=0.05cm,style={dotted, right}](v2)(v3)
		\Edge[lw=0.05cm,style={dotted, right}](v3)(v4)
		\Edge[lw=0.05cm,style={dotted, right}](v1)(v5)
	
		\Edge[lw=0.1cm,style={post, bend right}](v3)(v0)
		\Edge[lw=0.05cm,style={dotted, right}](v5)(x)
		
		\end{tikzpicture}
		
		\caption{Partial Configuration with $r'>r+1$ and $i>r'-r+1$}
		\label{fig:partialconfiguration}
	\end{figure}
	
	So now we consider the case $r'=r$. If $i=r$, one has $\deg(v) \le 2(n-1)-(2r-2)$ again.\\
	If $1<i<r$, analogously as before by picking a vertex $x$ such that $d(u_i,x) \ge r$, again we conclude $\deg(v) \le 2(n-1)-(2r-2)$.
	For this, we note that $\vc{u_j v}\not\in A$ when $1 \le j \le i-1,$
	$\vc{vu_j}\not \in A$ for $2 \le j \le r$ and $\vc{vy} \not \in A$ for at least $r-i$ vertices $y$ belonging to the shortest path from $v$ to $x$ for which $d(v,y) \ge i$ (the latter implying that these are disjoint from the $u_j$). \\
	In the remaining case with $i=1$, equality can occur only if the digraph satisfies the conditions mentioned in this proposition.
	Here $w_r=x$ is the vertex such that $d(u_1,x)=r.$	
\end{proof}

\begin{thr}\label{mainVz_bicDi}
	For $r\ge 4$, there exists a value $n_1(r)$ such that for all $n > n_1(r)$ the following hold
	\begin{itemize}
		\item for any biconnected digraph $D$ of order $n$ with outradius $r$, we have $\lvert A(D) \rvert\le \lvert A(D_{n,r,1})\rvert =(n-(r-1))^2+(r-3)$, with equality if and only if $D \cong D_{n,r,s}$ for $1 \le s \le \frac{n-2r+2}{2}$.
	\end{itemize}
\end{thr}

\begin{proof}
	Let $n_1:=n_1(r)=(2r-2)n_0.$
	Take a digraph $D$ of size $n > n_1$ which has the maximum size among all digraphs of order $n$ and outradius $r.$
	By Lemma~\ref{lem2di_size}, iteratively one can remove a vertex $v_n, v_{n-1}, \ldots, v_{n_0+1}$ from $D$ without changing outradius nor remaining distances.
	If for every of these $n-n_0$ removals, there is no equality in Proposition~\ref{proplem_size}, then as $n>(2r-2)n_0$ we have
	\begin{align*}
	\lvert A(D) \rvert \le n_0(n_0-1)+\sum_{i=n_0+1}^n 2(i-1)-(2r-2) &=n^2-(2r-1)n+(2r-2)n_0\\ &< \lvert A(D_{n,r,1})\rvert.
	\end{align*} 
	So now assume there was equality in some step in Lemma~\ref{lem2di_size}. Then we know a partial characterization of the digraph $D_m=(V_m,A_m)=D \backslash \{v_n,v_{n-1} \ldots v_{m+1}\}$ in that step.
	Let $B=\{u_1,u_2, \ldots, u_{r-1}, u_r, v, w_2, \ldots ,w_r\}$ and $R=V_m \backslash B$.
	Note that there is no directed edge from $u_i$ to $w_j$ if $1 \le i \le r-1$ and $2 \le j \le r$ as otherwise $\ecc^+(u_i)<r$.
	Similarly, there is no directed edge from $w_i$ to $u_j$ if $2 \le i \le r$ and $2 \le j \le r-1$ and at most one to $u_1$ or $u_r$.
	Since $\ecc^+(u_r) \ge r$, we need $d(u_r,u_{r-1})=r$ or $d(u_r,w_r)=r$.
	This implies that it is impossible there are directed edges from $u_r$ to both some $w_i$ and $u_j$ with $2 \le i \le r$ and $1 \le j \le r-1.$
	If $N^+(u_r) \cap B = \{v,w_2,w_3, \ldots, w_r\}$, then $\ecc^+(u_{r-1})<r$ which is a contradiction again.
	Knowing all of this, we already have an upper bound on the number of directed edges between vertices in $B$ which is equal to $\lvert A(D_{2r,r,1})\rvert.$ Equality is possible if and only if $D[B]$ is isomorphic to $D_{2r,r,1}.$ 
	For this, note that if all equalities need to hold, we cannot have a directed edge from some $w_i$ (with $2 \le i \le r$) to $u_r$, since then $d(w_i,u_{r-1})=2$ and hence $\ecc^+(w_i)<r$.\\		
	Next, let $z \in R$ be a remaining vertex.
	We now prove that there cannot be more than $2r+3$ directed edges between $z$ and $B.$
	Note that $N^+(z) \cap B \subset \{w_3,w_2,v,u_1,u_2\}$.
	If $w_3 \in N^+(z)$, then there is no $u_i$ with $1 \le i \le r-1$ in $N^-(z)$ and hence there are at most $r+6<2r+3$ directed edges between $z$ and $B$ in this case.
	Since $w_2$ and $u_2$ cannot be both in $N^+(z)$ as otherwise $\ecc^+(z)<r$, there are indeed at most $2r+3$ directed edges between $z$ and $B$. In those cases $N^+(z) \cap B$ equals $\{w_2,v,u_1\}$ or $\{v,u_1,u_2\}$.
	So indeed $ \lvert A_m \rvert \le \lvert A(D_{m,r,1}) \rvert$ with equality if and only if $D_m$ is isomorphic to $D_{m,r,s}$ for a certain $s$.\\
	For every vertex which is added, we know the total degree of that vertex is bounded again by Proposition~\ref{proplem_size} and equality occurs if and only if the digraph with this additional vertex is also of the form $D_{m',r,s'}$ for certain $m'$ and $s'$, from which the conclusion follows.
\end{proof}

\subsection{Maximum size given radius}\label{subsec:maxMgivenRad}

Theorem~\ref{thr:Vz_rad+3} tells us that the digraphs with order $n$ and outradius $3$ are exactly the digraphs conjectured to minimize the total distance in~\cite[Conj.~1.3]{cambie2019extremal}.
That same question of minimizing the total distance given order and radius instead of outradius may be harder, and we observe that the extremal digraphs are sometimes different than the ones maximizing the size given order and radius.
When the radius is at most $2$, the extremal digraphs are the same as in \cite[Prop.~5.1]{cambie2019extremal}. For this, it is sufficient to see that the extremal digraphs never will have diameter $3$ and so maximizing size and minimizing the total distance indeed correspond with each other.
Doing some computer simulation for $n = 6$, we observe that the extremal digraphs are the same except for $\rad =2.5$ and $\rad=3$.

The two digraphs at the left of Figure~\ref{fig:digraph_rad2.5_SmallWm} are the extremal graphs attaining the minimum Wiener index of $45$ among all digraphs of order $n$ and radius $\frac 52.$ Their size equals $18$.
The digraph at the right of Figure~\ref{fig:digraph_rad2.5_SmallWm} is the unique extremal digraph among all digraphs of order $n$ and radius $\frac 52$ attaining the maximum size, being $20.$
We note that this digraph is the one maximizing the size for diameter $5.$

\begin{figure}[h]
	\centering
	
	\begin{tikzpicture}
	\definecolor{cv0}{rgb}{0.0,0.0,0.0}
	\definecolor{cfv0}{rgb}{1.0,1.0,1.0}
	\definecolor{clv0}{rgb}{0.0,0.0,0.0}
	\definecolor{cv1}{rgb}{0.0,0.0,0.0}
	\definecolor{cfv1}{rgb}{1.0,1.0,1.0}
	\definecolor{clv1}{rgb}{0.0,0.0,0.0}
	\definecolor{cv2}{rgb}{0.0,0.0,0.0}
	\definecolor{cfv2}{rgb}{1.0,1.0,1.0}
	\definecolor{clv2}{rgb}{0.0,0.0,0.0}
	\definecolor{cv3}{rgb}{0.0,0.0,0.0}
	\definecolor{cfv3}{rgb}{1.0,1.0,1.0}
	\definecolor{clv3}{rgb}{0.0,0.0,0.0}
	\definecolor{cv4}{rgb}{0.0,0.0,0.0}
	\definecolor{cfv4}{rgb}{1.0,1.0,1.0}
	\definecolor{clv4}{rgb}{0.0,0.0,0.0}
	\definecolor{cv5}{rgb}{0.0,0.0,0.0}
	\definecolor{cfv5}{rgb}{1.0,1.0,1.0}
	\definecolor{clv5}{rgb}{0.0,0.0,0.0}
	\definecolor{cv0v1}{rgb}{0.0,0.0,0.0}
	\definecolor{cv0v3}{rgb}{0.0,0.0,0.0}
	\definecolor{cv1v0}{rgb}{0.0,0.0,0.0}
	\definecolor{cv1v2}{rgb}{0.0,0.0,0.0}
	\definecolor{cv2v0}{rgb}{0.0,0.0,0.0}
	\definecolor{cv2v1}{rgb}{0.0,0.0,0.0}
	\definecolor{cv2v3}{rgb}{0.0,0.0,0.0}
	\definecolor{cv2v5}{rgb}{0.0,0.0,0.0}
	\definecolor{cv3v0}{rgb}{0.0,0.0,0.0}
	\definecolor{cv3v4}{rgb}{0.0,0.0,0.0}
	\definecolor{cv4v0}{rgb}{0.0,0.0,0.0}
	\definecolor{cv4v1}{rgb}{0.0,0.0,0.0}
	\definecolor{cv4v3}{rgb}{0.0,0.0,0.0}
	\definecolor{cv4v5}{rgb}{0.0,0.0,0.0}
	\definecolor{cv5v1}{rgb}{0.0,0.0,0.0}
	\definecolor{cv5v2}{rgb}{0.0,0.0,0.0}
	\definecolor{cv5v3}{rgb}{0.0,0.0,0.0}
	\definecolor{cv5v4}{rgb}{0.0,0.0,0.0}
	\Vertex[style={minimum size=1.0cm,draw=cv0,fill=cfv0,text=clv0,shape=circle},LabelOut=false,L=\hbox{$0$},x=0cm,y=0.0cm]{v0}
	\Vertex[style={minimum size=1.0cm,draw=cv1,fill=cfv1,text=clv1,shape=circle},LabelOut=false,L=\hbox{$1$},x=0cm,y=1.75cm]{v1}
	\Vertex[style={minimum size=1.0cm,draw=cv2,fill=cfv2,text=clv2,shape=circle},LabelOut=false,L=\hbox{$2$},x=1.25cm,y=3cm]{v2}
	\Vertex[style={minimum size=1.0cm,draw=cv3,fill=cfv3,text=clv3,shape=circle},LabelOut=false,L=\hbox{$3$},x=3cm,y=1.25cm]{v3}
	\Vertex[style={minimum size=1.0cm,draw=cv4,fill=cfv4,text=clv4,shape=circle},LabelOut=false,L=\hbox{$4$},x=1.75cm,y=0cm]{v4}
	\Vertex[style={minimum size=1.0cm,draw=cv5,fill=cfv5,text=clv5,shape=circle},LabelOut=false,L=\hbox{$5$},x=3cm,y=3cm]{v5}
	\Edge[lw=0.1cm,style={post, color=cv0v1,},](v0)(v1)
	\Edge[lw=0.1cm,style={post, color=cv0v3,},](v0)(v3)
	\Edge[lw=0.1cm,style={post, color=cv1v0,},](v1)(v0)
	\Edge[lw=0.1cm,style={post, color=cv1v2,},](v1)(v2)
	\Edge[lw=0.1cm,style={post, color=cv2v0,},](v2)(v0)
	\Edge[lw=0.1cm,style={post, color=cv2v1,},](v2)(v1)
	\Edge[lw=0.1cm,style={post, color=cv2v3,},](v2)(v3)
	\Edge[lw=0.1cm,style={post, color=cv2v5,},](v2)(v5)
	\Edge[lw=0.1cm,style={post, color=cv3v0,},](v3)(v0)
	\Edge[lw=0.1cm,style={post, color=cv3v4,},](v3)(v4)
	\Edge[lw=0.1cm,style={post, color=cv4v0,},](v4)(v0)
	\Edge[lw=0.1cm,style={post, color=cv4v1,},](v4)(v1)
	\Edge[lw=0.1cm,style={post, color=cv4v3,},](v4)(v3)
	\Edge[lw=0.1cm,style={post, color=cv4v5,},](v4)(v5)
	\Edge[lw=0.1cm,style={post, color=cv5v1,},](v5)(v1)
	\Edge[lw=0.1cm,style={post, color=cv5v2,},](v5)(v2)
	\Edge[lw=0.1cm,style={post, color=cv5v3,},](v5)(v3)
	\Edge[lw=0.1cm,style={post, color=cv5v4,},](v5)(v4)
	
	\end{tikzpicture}
	\quad
	\begin{tikzpicture}
	\definecolor{cv0}{rgb}{0.0,0.0,0.0}
	\definecolor{cfv0}{rgb}{1.0,1.0,1.0}
	\definecolor{clv0}{rgb}{0.0,0.0,0.0}
	\definecolor{cv1}{rgb}{0.0,0.0,0.0}
	\definecolor{cfv1}{rgb}{1.0,1.0,1.0}
	\definecolor{clv1}{rgb}{0.0,0.0,0.0}
	\definecolor{cv2}{rgb}{0.0,0.0,0.0}
	\definecolor{cfv2}{rgb}{1.0,1.0,1.0}
	\definecolor{clv2}{rgb}{0.0,0.0,0.0}
	\definecolor{cv3}{rgb}{0.0,0.0,0.0}
	\definecolor{cfv3}{rgb}{1.0,1.0,1.0}
	\definecolor{clv3}{rgb}{0.0,0.0,0.0}
	\definecolor{cv4}{rgb}{0.0,0.0,0.0}
	\definecolor{cfv4}{rgb}{1.0,1.0,1.0}
	\definecolor{clv4}{rgb}{0.0,0.0,0.0}
	\definecolor{cv5}{rgb}{0.0,0.0,0.0}
	\definecolor{cfv5}{rgb}{1.0,1.0,1.0}
	\definecolor{clv5}{rgb}{0.0,0.0,0.0}
	\definecolor{cv0v1}{rgb}{0.0,0.0,0.0}
	\definecolor{cv0v3}{rgb}{0.0,0.0,0.0}
	\definecolor{cv0v4}{rgb}{0.0,0.0,0.0}
	\definecolor{cv0v5}{rgb}{0.0,0.0,0.0}
	\definecolor{cv1v2}{rgb}{0.0,0.0,0.0}
	\definecolor{cv1v3}{rgb}{0.0,0.0,0.0}
	\definecolor{cv1v4}{rgb}{0.0,0.0,0.0}
	\definecolor{cv1v5}{rgb}{0.0,0.0,0.0}
	\definecolor{cv2v0}{rgb}{0.0,0.0,0.0}
	\definecolor{cv2v3}{rgb}{0.0,0.0,0.0}
	\definecolor{cv2v4}{rgb}{0.0,0.0,0.0}
	\definecolor{cv2v5}{rgb}{0.0,0.0,0.0}
	\definecolor{cv3v0}{rgb}{0.0,0.0,0.0}
	\definecolor{cv3v4}{rgb}{0.0,0.0,0.0}
	\definecolor{cv4v1}{rgb}{0.0,0.0,0.0}
	\definecolor{cv4v5}{rgb}{0.0,0.0,0.0}
	\definecolor{cv5v2}{rgb}{0.0,0.0,0.0}
	\definecolor{cv5v3}{rgb}{0.0,0.0,0.0}
	\Vertex[style={minimum size=1.0cm,draw=cv0,fill=cfv0,text=clv0,shape=circle},LabelOut=false,L=\hbox{$0$},x=0cm,y=0.0cm]{v0}
	\Vertex[style={minimum size=1.0cm,draw=cv1,fill=cfv1,text=clv1,shape=circle},LabelOut=false,L=\hbox{$1$},x=0cm,y=1.75cm]{v1}
	\Vertex[style={minimum size=1.0cm,draw=cv2,fill=cfv2,text=clv2,shape=circle},LabelOut=false,L=\hbox{$2$},x=1.25cm,y=3cm]{v2}
	\Vertex[style={minimum size=1.0cm,draw=cv3,fill=cfv3,text=clv3,shape=circle},LabelOut=false,L=\hbox{$3$},x=3cm,y=1.25cm]{v3}
	\Vertex[style={minimum size=1.0cm,draw=cv4,fill=cfv4,text=clv4,shape=circle},LabelOut=false,L=\hbox{$4$},x=1.75cm,y=0cm]{v4}
	\Vertex[style={minimum size=1.0cm,draw=cv5,fill=cfv5,text=clv5,shape=circle},LabelOut=false,L=\hbox{$5$},x=3cm,y=3cm]{v5}
	\Edge[lw=0.1cm,style={post,color=cv0v1,},](v0)(v1)
	\Edge[lw=0.1cm,style={post,color=cv0v3,},](v0)(v3)
	\Edge[lw=0.1cm,style={post,color=cv0v4,},](v0)(v4)
	\Edge[lw=0.1cm,style={post,color=cv0v5,},](v0)(v5)
	\Edge[lw=0.1cm,style={post,color=cv1v2,},](v1)(v2)
	\Edge[lw=0.1cm,style={post,color=cv1v3,},](v1)(v3)
	\Edge[lw=0.1cm,style={post,color=cv1v4,},](v1)(v4)
	\Edge[lw=0.1cm,style={post,color=cv1v5,},](v1)(v5)
	\Edge[lw=0.1cm,style={post,color=cv2v0,},](v2)(v0)
	\Edge[lw=0.1cm,style={post,color=cv2v3,},](v2)(v3)
	\Edge[lw=0.1cm,style={post,color=cv2v4,},](v2)(v4)
	\Edge[lw=0.1cm,style={post,color=cv2v5,},](v2)(v5)
	\Edge[lw=0.1cm,style={post,color=cv3v0,},](v3)(v0)
	\Edge[lw=0.1cm,style={post,color=cv3v4,},](v3)(v4)
	\Edge[lw=0.1cm,style={post,color=cv4v1,},](v4)(v1)
	\Edge[lw=0.1cm,style={post,color=cv4v5,},](v4)(v5)
	\Edge[lw=0.1cm,style={post,color=cv5v2,},](v5)(v2)
	\Edge[lw=0.1cm,style={post,color=cv5v3,},](v5)(v3)
	\end{tikzpicture}
	\quad
	\begin{tikzpicture}
	\definecolor{cv0}{rgb}{0.0,0.0,0.0}
	\definecolor{cfv0}{rgb}{1.0,1.0,1.0}
	\definecolor{clv0}{rgb}{0.0,0.0,0.0}
	\definecolor{cv1}{rgb}{0.0,0.0,0.0}
	\definecolor{cfv1}{rgb}{1.0,1.0,1.0}
	\definecolor{clv1}{rgb}{0.0,0.0,0.0}
	\definecolor{cv2}{rgb}{0.0,0.0,0.0}
	\definecolor{cfv2}{rgb}{1.0,1.0,1.0}
	\definecolor{clv2}{rgb}{0.0,0.0,0.0}
	\definecolor{cv3}{rgb}{0.0,0.0,0.0}
	\definecolor{cfv3}{rgb}{1.0,1.0,1.0}
	\definecolor{clv3}{rgb}{0.0,0.0,0.0}
	\definecolor{cv4}{rgb}{0.0,0.0,0.0}
	\definecolor{cfv4}{rgb}{1.0,1.0,1.0}
	\definecolor{clv4}{rgb}{0.0,0.0,0.0}
	\definecolor{cv5}{rgb}{0.0,0.0,0.0}
	\definecolor{cfv5}{rgb}{1.0,1.0,1.0}
	\definecolor{clv5}{rgb}{0.0,0.0,0.0}
	\definecolor{cv0v4}{rgb}{0.0,0.0,0.0}
	\definecolor{cv1v0}{rgb}{0.0,0.0,0.0}
	\definecolor{cv1v2}{rgb}{0.0,0.0,0.0}
	\definecolor{cv1v3}{rgb}{0.0,0.0,0.0}
	\definecolor{cv1v4}{rgb}{0.0,0.0,0.0}
	\definecolor{cv1v5}{rgb}{0.0,0.0,0.0}
	\definecolor{cv2v0}{rgb}{0.0,0.0,0.0}
	\definecolor{cv2v3}{rgb}{0.0,0.0,0.0}
	\definecolor{cv2v4}{rgb}{0.0,0.0,0.0}
	\definecolor{cv2v5}{rgb}{0.0,0.0,0.0}
	\definecolor{cv3v0}{rgb}{0.0,0.0,0.0}
	\definecolor{cv3v1}{rgb}{0.0,0.0,0.0}
	\definecolor{cv3v2}{rgb}{0.0,0.0,0.0}
	\definecolor{cv3v4}{rgb}{0.0,0.0,0.0}
	\definecolor{cv3v5}{rgb}{0.0,0.0,0.0}
	\definecolor{cv4v0}{rgb}{0.0,0.0,0.0}
	\definecolor{cv4v5}{rgb}{0.0,0.0,0.0}
	\definecolor{cv5v0}{rgb}{0.0,0.0,0.0}
	\definecolor{cv5v2}{rgb}{0.0,0.0,0.0}
	\definecolor{cv5v4}{rgb}{0.0,0.0,0.0}
	\Vertex[style={minimum size=1.0cm,draw=cv0,fill=cfv0,text=clv0,shape=circle},LabelOut=false,L=\hbox{$0$},x=0cm,y=0.0cm]{v0}
	\Vertex[style={minimum size=1.0cm,draw=cv1,fill=cfv1,text=clv1,shape=circle},LabelOut=false,L=\hbox{$1$},x=0cm,y=1.75cm]{v4}
	\Vertex[style={minimum size=1.0cm,draw=cv2,fill=cfv2,text=clv2,shape=circle},LabelOut=false,L=\hbox{$2$},x=1.25cm,y=3cm]{v5}
	\Vertex[style={minimum size=1.0cm,draw=cv3,fill=cfv3,text=clv3,shape=circle},LabelOut=false,L=\hbox{$4$},x=3cm,y=1.25cm]{v3}
	\Vertex[style={minimum size=1.0cm,draw=cv4,fill=cfv4,text=clv4,shape=circle},LabelOut=false,L=\hbox{$5$},x=1.75cm,y=0cm]{v1}
	\Vertex[style={minimum size=1.0cm,draw=cv5,fill=cfv5,text=clv5,shape=circle},LabelOut=false,L=\hbox{$3$},x=3cm,y=3cm]{v2}
	\Edge[lw=0.1cm,style={post,color=cv0v4,},](v0)(v4)
	\Edge[lw=0.1cm,style={post,color=cv1v0,},](v1)(v0)
	\Edge[lw=0.1cm,style={post,color=cv1v2,},](v1)(v2)
	\Edge[lw=0.1cm,style={post,color=cv1v3,},](v1)(v3)
	\Edge[lw=0.1cm,style={post,color=cv1v4,},](v1)(v4)
	\Edge[lw=0.1cm,style={post,color=cv1v5,},](v1)(v5)
	\Edge[lw=0.1cm,style={post,color=cv2v0,},](v2)(v0)
	\Edge[lw=0.1cm,style={post,color=cv2v3,},](v2)(v3)
	\Edge[lw=0.1cm,style={post,color=cv2v4,},](v2)(v4)
	\Edge[lw=0.1cm,style={post,color=cv2v5,},](v2)(v5)
	\Edge[lw=0.1cm,style={post,color=cv3v0,},](v3)(v0)
	\Edge[lw=0.1cm,style={post,color=cv3v1,},](v3)(v1)
	\Edge[lw=0.1cm,style={post,color=cv3v2,},](v3)(v2)
	\Edge[lw=0.1cm,style={post,color=cv3v4,},](v3)(v4)
	\Edge[lw=0.1cm,style={post,color=cv3v5,},](v3)(v5)
	\Edge[lw=0.1cm,style={post,color=cv4v0,},](v4)(v0)
	\Edge[lw=0.1cm,style={post,color=cv4v5,},](v4)(v5)
	\Edge[lw=0.1cm,style={post,color=cv5v0,},](v5)(v0)
	\Edge[lw=0.1cm,style={post,color=cv5v2,},](v5)(v2)
	\Edge[lw=0.1cm,style={post,color=cv5v4,},](v5)(v4)
	\end{tikzpicture}

	\caption{Digraphs with $\rad=2.5$ and minimal Wiener index (2 extremal digraphs), resp. maximal size for $n= 6$}
	\label{fig:digraph_rad2.5_SmallWm}
\end{figure}
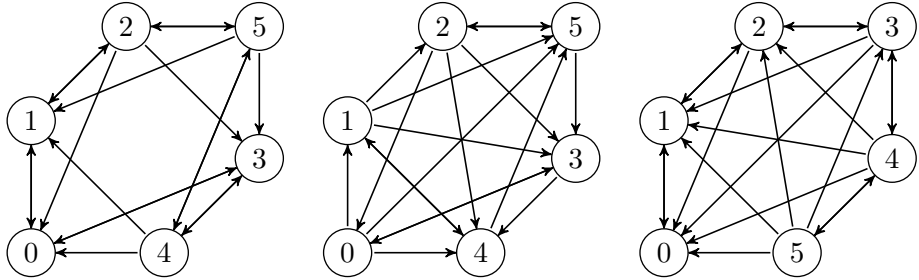

\begin{figure}[h]
	\centering

	\begin{tikzpicture}
	\definecolor{cv0}{rgb}{0.0,0.0,0.0}
	\definecolor{cfv0}{rgb}{1.0,1.0,1.0}
	\definecolor{clv0}{rgb}{0.0,0.0,0.0}
	\definecolor{cv1}{rgb}{0.0,0.0,0.0}
	\definecolor{cfv1}{rgb}{1.0,1.0,1.0}
	\definecolor{clv1}{rgb}{0.0,0.0,0.0}
	\definecolor{cv2}{rgb}{0.0,0.0,0.0}
	\definecolor{cfv2}{rgb}{1.0,1.0,1.0}
	\definecolor{clv2}{rgb}{0.0,0.0,0.0}
	\definecolor{cv3}{rgb}{0.0,0.0,0.0}
	\definecolor{cfv3}{rgb}{1.0,1.0,1.0}
	\definecolor{clv3}{rgb}{0.0,0.0,0.0}
	\definecolor{cv4}{rgb}{0.0,0.0,0.0}
	\definecolor{cfv4}{rgb}{1.0,1.0,1.0}
	\definecolor{clv4}{rgb}{0.0,0.0,0.0}
	\definecolor{cv5}{rgb}{0.0,0.0,0.0}
	\definecolor{cfv5}{rgb}{1.0,1.0,1.0}
	\definecolor{clv5}{rgb}{0.0,0.0,0.0}
	\definecolor{cv0v1}{rgb}{0.0,0.0,0.0}
	\definecolor{cv0v2}{rgb}{0.0,0.0,0.0}
	\definecolor{cv0v3}{rgb}{0.0,0.0,0.0}
	\definecolor{cv0v4}{rgb}{0.0,0.0,0.0}
	\definecolor{cv1v3}{rgb}{0.0,0.0,0.0}
	\definecolor{cv2v0}{rgb}{0.0,0.0,0.0}
	\definecolor{cv2v1}{rgb}{0.0,0.0,0.0}
	\definecolor{cv2v4}{rgb}{0.0,0.0,0.0}
	\definecolor{cv3v1}{rgb}{0.0,0.0,0.0}
	\definecolor{cv3v4}{rgb}{0.0,0.0,0.0}
	\definecolor{cv3v5}{rgb}{0.0,0.0,0.0}
	\definecolor{cv4v2}{rgb}{0.0,0.0,0.0}
	\definecolor{cv5v1}{rgb}{0.0,0.0,0.0}
	\definecolor{cv5v2}{rgb}{0.0,0.0,0.0}
	\definecolor{cv5v3}{rgb}{0.0,0.0,0.0}
	\definecolor{cv5v4}{rgb}{0.0,0.0,0.0}
	\Vertex[style={minimum size=1.0cm,draw=cv0,fill=cfv0,text=clv0,shape=circle},LabelOut=false,L=\hbox{$0$},x=0cm,y=0.0cm]{v0}
	\Vertex[style={minimum size=1.0cm,draw=cv1,fill=cfv1,text=clv1,shape=circle},LabelOut=false,L=\hbox{$1$},x=4.0cm,y=0cm]{v1}
	\Vertex[style={minimum size=1.0cm,draw=cv2,fill=cfv2,text=clv2,shape=circle},LabelOut=false,L=\hbox{$2$},x=2cm,y=0.6cm]{v2}
	\Vertex[style={minimum size=1.0cm,draw=cv3,fill=cfv3,text=clv3,shape=circle},LabelOut=false,L=\hbox{$4$},x=2cm,y=1.9cm]{v3}
	\Vertex[style={minimum size=1.0cm,draw=cv4,fill=cfv4,text=clv4,shape=circle},LabelOut=false,L=\hbox{$3$},x=0.0cm,y=2.5cm]{v4}
	\Vertex[style={minimum size=1.0cm,draw=cv5,fill=cfv5,text=clv5,shape=circle},LabelOut=false,L=\hbox{$5$},x=4cm,y=2.5cm]{v5}
	\Edge[lw=0.1cm,style={post, color=cv0v1,},](v0)(v1)
	\Edge[lw=0.1cm,style={post, color=cv0v2,},](v0)(v2)
	\Edge[lw=0.1cm,style={post, color=cv0v3,},](v0)(v3)
	\Edge[lw=0.1cm,style={post, color=cv0v4,},](v0)(v4)
	\Edge[lw=0.1cm,style={post, color=cv1v3,},](v1)(v3)
	\Edge[lw=0.1cm,style={post, color=cv2v0,},](v2)(v0)
	\Edge[lw=0.1cm,style={post, color=cv2v1,},](v2)(v1)
	\Edge[lw=0.1cm,style={post, color=cv2v4,},](v2)(v4)
	\Edge[lw=0.1cm,style={post, color=cv3v1,},](v3)(v1)
	\Edge[lw=0.1cm,style={post, color=cv3v4,},](v3)(v4)
	\Edge[lw=0.1cm,style={post, color=cv3v5,},](v3)(v5)
	\Edge[lw=0.1cm,style={post, color=cv4v2,},](v4)(v2)
	\Edge[lw=0.1cm,style={post, color=cv5v1,},](v5)(v1)
	\Edge[lw=0.1cm,style={post, color=cv5v2,},](v5)(v2)
	\Edge[lw=0.1cm,style={post, color=cv5v3,},](v5)(v3)
	\Edge[lw=0.1cm,style={post, color=cv5v4,},](v5)(v4)
	\end{tikzpicture}
	\quad 
	\begin{tikzpicture}
	\definecolor{cv0}{rgb}{0.0,0.0,0.0}
	\definecolor{cfv0}{rgb}{1.0,1.0,1.0}
	\definecolor{clv0}{rgb}{0.0,0.0,0.0}
	\definecolor{cv1}{rgb}{0.0,0.0,0.0}
	\definecolor{cfv1}{rgb}{1.0,1.0,1.0}
	\definecolor{clv1}{rgb}{0.0,0.0,0.0}
	\definecolor{cv2}{rgb}{0.0,0.0,0.0}
	\definecolor{cfv2}{rgb}{1.0,1.0,1.0}
	\definecolor{clv2}{rgb}{0.0,0.0,0.0}
	\definecolor{cv3}{rgb}{0.0,0.0,0.0}
	\definecolor{cfv3}{rgb}{1.0,1.0,1.0}
	\definecolor{clv3}{rgb}{0.0,0.0,0.0}
	\definecolor{cv4}{rgb}{0.0,0.0,0.0}
	\definecolor{cfv4}{rgb}{1.0,1.0,1.0}
	\definecolor{clv4}{rgb}{0.0,0.0,0.0}
	\definecolor{cv5}{rgb}{0.0,0.0,0.0}
	\definecolor{cfv5}{rgb}{1.0,1.0,1.0}
	\definecolor{clv5}{rgb}{0.0,0.0,0.0}
	\definecolor{cv0v2}{rgb}{0.0,0.0,0.0}
	\definecolor{cv0v5}{rgb}{0.0,0.0,0.0}
	\definecolor{cv1v0}{rgb}{0.0,0.0,0.0}
	\definecolor{cv1v4}{rgb}{0.0,0.0,0.0}
	\definecolor{cv1v5}{rgb}{0.0,0.0,0.0}
	\definecolor{cv2v0}{rgb}{0.0,0.0,0.0}
	\definecolor{cv2v3}{rgb}{0.0,0.0,0.0}
	\definecolor{cv2v4}{rgb}{0.0,0.0,0.0}
	\definecolor{cv2v5}{rgb}{0.0,0.0,0.0}
	\definecolor{cv3v0}{rgb}{0.0,0.0,0.0}
	\definecolor{cv3v1}{rgb}{0.0,0.0,0.0}
	\definecolor{cv3v2}{rgb}{0.0,0.0,0.0}
	\definecolor{cv3v4}{rgb}{0.0,0.0,0.0}
	\definecolor{cv3v5}{rgb}{0.0,0.0,0.0}
	\definecolor{cv4v1}{rgb}{0.0,0.0,0.0}
	\definecolor{cv4v5}{rgb}{0.0,0.0,0.0}
	\definecolor{cv5v4}{rgb}{0.0,0.0,0.0}
	\Vertex[style={minimum size=1.0cm,draw=cv0,fill=cfv0,text=clv0,shape=circle},LabelOut=false,L=\hbox{$0$},x=0cm,y=0.0cm]{v0}
	\Vertex[style={minimum size=1.0cm,draw=cv1,fill=cfv1,text=clv1,shape=circle},LabelOut=false,L=\hbox{$1$},x=4.0cm,y=0cm]{v1}
	\Vertex[style={minimum size=1.0cm,draw=cv2,fill=cfv2,text=clv2,shape=circle},LabelOut=false,L=\hbox{$2$},x=2cm,y=0.6cm]{v3}
	\Vertex[style={minimum size=1.0cm,draw=cv3,fill=cfv3,text=clv3,shape=circle},LabelOut=false,L=\hbox{$4$},x=2cm,y=1.9cm]{v5}
	\Vertex[style={minimum size=1.0cm,draw=cv4,fill=cfv4,text=clv4,shape=circle},LabelOut=false,L=\hbox{$3$},x=0.0cm,y=2.5cm]{v2}
	\Vertex[style={minimum size=1.0cm,draw=cv5,fill=cfv5,text=clv5,shape=circle},LabelOut=false,L=\hbox{$5$},x=4cm,y=2.5cm]{v4}
	\Edge[lw=0.1cm,style={post,color=cv0v2,},](v0)(v2)
	\Edge[lw=0.1cm,style={post,color=cv0v5,},](v0)(v5)
	\Edge[lw=0.1cm,style={post,color=cv1v0,},](v1)(v0)
	\Edge[lw=0.1cm,style={post,color=cv1v4,},](v1)(v4)
	\Edge[lw=0.1cm,style={post,color=cv1v5,},](v1)(v5)
	\Edge[lw=0.1cm,style={post,color=cv2v0,},](v2)(v0)
	\Edge[lw=0.1cm,style={post,color=cv2v3,},](v2)(v3)
	\Edge[lw=0.1cm,style={post,color=cv2v4,},](v2)(v4)
	\Edge[lw=0.1cm,style={post,color=cv2v5,},](v2)(v5)
	\Edge[lw=0.1cm,style={post,color=cv3v0,},](v3)(v0)
	\Edge[lw=0.1cm,style={post,color=cv3v1,},](v3)(v1)
	\Edge[lw=0.1cm,style={post,color=cv3v2,},](v3)(v2)
	\Edge[lw=0.1cm,style={post,color=cv3v4,},](v3)(v4)
	\Edge[lw=0.1cm,style={post,color=cv3v5,},](v3)(v5)
	\Edge[lw=0.1cm,style={post,color=cv4v1,},](v4)(v1)
	\Edge[lw=0.1cm,style={post,color=cv4v5,},](v4)(v5)
	\Edge[lw=0.1cm,style={post,color=cv5v4,},](v5)(v4)
	\end{tikzpicture}

	\caption{Digraphs with $\rad=3$ and minimal Wiener index, resp. maximal size for $n= 6$}
	\label{fig:digraph_rad_smallWm}
\end{figure}
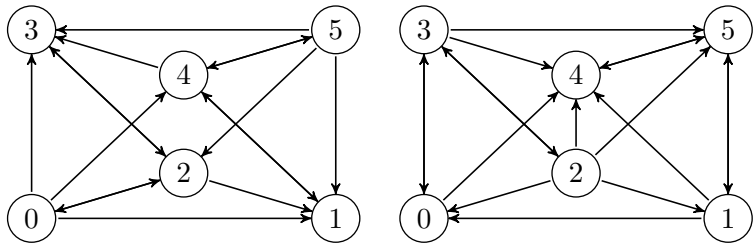

For a radius of $3$, the left digraph in Figure~\ref{fig:digraph_rad_smallWm} has size $m=16$ and total distance $W=52$, while the digraph at the right has size $m=17$ and total distance $W=54.$ 

So the digraphs with maximum size and those with minimum total distance given the conditions on order and radius, are pretty different in these cases.


For $n=5$ and $rad=2.5$ there are $7$ different digraphs having the maximum size of $11$.
So for small values of $n$, there is sometimes no clear structure.
On the other hand, the structure of the extremal digraphs for $n \ge 2r+1$ is conjectured below. This is the analogue of Conjecture~\ref{VZdi} for radius instead of outradius.

\begin{conj}\label{conj:VZdi_rad}
	Let $r \ge \frac 52.$
	For $n \ge 2r+1$, the biconnected digraphs with order $n$ and $\rad=r$ that maximizes the size are formed by taking two blow-ups at consecutive, non-end vertices of the digraph $\overline{\Gamma}_{2r+1}.$
	Furthermore, the extremal digraphs are exactly those of the form $\overline{\Gamma}_{n,2r,i,s}$ where $1 \le i \le 2r-2$ and $1\le s \le n-2r.$ 
\end{conj}

We prove this conjecture for the case $n=2r+1.$

\begin{prop}
	Let $D$ be a graph with order $n=2r+1$ and radius $r \ge \frac 52$. Then its size $\lvert A(D) \rvert$ is bounded by 
	$$\lvert A(\overline{\Gamma}_{2r+1}) \rvert=(2r+1)(r+1)-1.$$
	and equality does hold if and only if $D$ is isomorphic to $\overline{\Gamma}_{2r+1}.$
\end{prop}

\begin{proof}
	As it was verified by a computer program for $r= \frac 52$ (can be done with more case distinction manually as well), we will assume $r\ge 3$.
	For every vertex $v$, one has $d(v,V)+d(V,v) \ge 2r$ by definition of the radius.
	Also there are no arcs from $v$ towards the vertices at distance $2,3, \ldots, d(v,V)$ or towards $v$ for at least $d(V,v)-1$ vertices.
	So at least $d(v,V)+d(V,v)-2 \ge 2r-2$ arcs from the possible $4r$ containing $x$ are not present, i.e. the total degree of every vertex is at most $2r+2.$
	This gives an upper bound of $\frac{n(2r+2)}{2}=(2r+1)(r+1).$
	Equality would be only possible if \begin{equation}\label{eq:1}
	d(v,V)+d(V,v) = 2r, \deg^+(v)=n-d(v,V) \mbox{ and }\deg^-(v)=n-d(V,v)
	\end{equation} for every vertex $v \in V$.
	Let $x$ and $y$ be vertices such that $d(x,y)$ equals the diameter which is at least $\lceil r \rceil \ge 3$.
	Now we can see that the indegree of $y$, $\deg^-(y)$, is exactly equal to one as there are arcs from $x$ towards every vertex not on the shortest path between $x$ and $y$ and $d(x,y) \ge 3.$
	But then $\deg^+(y)+\deg^-(y)\le 2r +1,$ so not all equalities in~\eqref{eq:1} can hold for $v=x.$
	From this we conclude that the upper bound $(2r+1)(r+1)$ can not be attained implying that $\lvert A(D) \rvert \le (2r+1)(r+1)-1.$
	
	%
	
	Now we will prove the uniqueness of the extremal digraph. 	
	Note that if the diameter equals $2r$, all vertices are part of the diameter and the digraph is a subdigraph of $\overline{\Gamma}_{2r+1}$, in which case the conclusion is clear.
	
	We will now do some case distinction, depending on the vertices that do not satisfy~\eqref{eq:1}. 
	First assume that there are only $2$ vertices $x$ and $y$ such that $d(x,y)$ equals the diameter and the constraints~\eqref{eq:1} are not satisfied by these two vertices.
	Observe that then the diameter is at least $4$ as it has to be strictly larger than $\lceil r \rceil \ge 3$.
	So for an outneighbour $z$ of $x$, we have $d(z,y)=d(z,V)$, while $z$ satisfies the constraints~\eqref{eq:1}.
	For every vertex $v$ not on the shortest path from $z$ to $y$, there is a directed edge from $z$ towards $v$ (since $\deg^+(z)=n-d(z,V)$) and as $d(z,y)\ge 3$ this implies $\deg^-(y)=1$. As a consequence, we need $\deg^+(y)=2r$, i.e. $d(y,V)=1.$
	Analogously for an inneighbour $w$ of $y$, and $v$ taken as before, there is a directed edge from $v$ to $w.$ 
	We conclude that $d(x,y) \le d(x,v)+d(v,y) \le 2+2=4.$
	But then $d(V,y)+d(y,V) \le 4+1<2r$, which is a contradiction.

	Now we look to the other case. Here we assume that $d(x,y)$ is the diameter and $x$ satisfies the equality constraints~\eqref{eq:1}. The case that $y$ satisfies the equality constraints~\eqref{eq:1} is analogous (e.g. by changing the direction of every arc).
	We know that $\deg^-(y)=1$ by the derivation of the upper bound.

	Since $\deg^+(y)+\deg^-(y)\ge 2r$, we have $\deg^+(y)\ge 2r-1$.
	This implies $\ecc^+(y)\le 2$ and consequently the diameter is at least $\ecc^-(y)\ge 2r-2.$

	If the diameter is $2r-2$, we need $\ecc^+(y)=2$ and thus $\deg^+(y) \le 2r-1,$ which needs to be an equality to ensure that the total degree of $y$ is $2r$. For all other vertices, the constraints~\eqref{eq:1} need to be satisfied.
	For the inneighbour $z$ from $y$, we know $d(V,z)=d(V,y)-1=2r-3$ since a shortest path from any vertex $v \in V \backslash \{z,y\}$ to $y$ uses $z$ and $d(y,z) \le 2$ as well.
	As $d(x,z) \ge 3$, we know $\deg^-(z)\le 2$ and so $\deg^-(z)=n-d(V,z)$ is not true.
	If the diameter is $2r-1$, one gets a contradiction by considering $y$ and its two predecessors $z',z$ in the shortest path from $x$ to $y$. In this case we have $d(V,z)=2r-2$ and $d(V,z')=2r-3$, while $\deg^-(z)\le 2$ and $\deg^-(z')\le 3.$ So also $z$ and $z'$ cannot satisfy~\eqref{eq:1}, from which the conclusion follows. 
\end{proof}

\section{Characterization of bipartite digraphs with maximum size}\label{sec:bipdigraph_maxSize}

In this section, we give the precise characterization of the bipartite digraphs with maximum size given their outradius. The maximum size was determined by Dankelmann~\cite{D15}, whose ideas we extend and write out in detail for completeness.

\begin{thr}\label{thr:bipdigraph_maxSize}
	Let $D$ be a bipartite digraph with outradius $r\ge 3$ and order $n$.
	Then $$\lvert A(D) \rvert \le \left\lceil \frac{n(n-2)}{4}\right\rceil +r-4 + \left\lfloor \frac{ (n-r+3)^2}{4} \right\rfloor.$$
	Furthermore equality occurs if and only if $D$ is isomorphic to the digraph obtained by taking blow-ups in $3$ consecutive non-end vertices of $\overline{\Gamma}^\star_{r+1} \cap K_{ \lceil \frac{r+1}2 \rceil, \lfloor \frac{r+1}2 \rfloor }$ by stable sets $aK_1,bK_1,cK_1$ with $\lvert a+c-(b+1) \rvert \le 1$ with the additional constraint that the bipartition class of the vertex with outeccentricity $r$ is at least as large as the other bipartition class in the case that $r$ is odd.
	Here $\overline{\Gamma}^\star_{r+1}=\overline{\Gamma}^\star_{r+1,r,1,1}$ is defined in Section~\ref{sec:intro} and the bipartition for $\overline{\Gamma}^\star_{r+1} \cap K_{ \lceil \frac{r+1}2 \rceil, \lfloor \frac{r+1}2 \rfloor }$ is the natural one with $\{v_i \colon 2 \mid i\} \cup \{v_i \colon  i \equiv 1 \pmod 2\}.$
	%
	%
\end{thr}
\begin{proof}
	Let $D$ be an extremal digraph.
	Let $v$ be a vertex of $D$ with $\ecc^+(v)=r$ and $N_i=N_i^+(v)=\{u \in V\mid d(v,u)=i\}$ be its $i^{th}$ outneighbourhood.
	Let $n_i=\lvert N_i \rvert$ for every $i$, $n_e=\sum_{i\mbox{ even}} n_i$ and $n_o=\sum_{i\mbox{ odd}} n_i$.
	Due to this definition and the bipartiteness condition, we have that for every $u \in N_i$ with $i$ even $N^+(u) \subset N_1 \cup N_3 \ldots \cup N_{i+1}$. 
	Similarly, for every $u \in N_i$ with $i$ odd, we have 
	$N^+(u) \subset N_0\cup N_2 \cup \ldots N_{i+1}.$
	Assume for some $u \in N_i$ with $i$ odd, equality is true.
	Then $\ecc^+(u) \le \max\{2,r-i\}<r$, since all vertices in $N_{ \le i+2}$ are at distance at most $2$ from $u.$
	
	Now assume $N^+(u) \subset N_1 \cup N_3 \ldots \cup N_{i+1}$ for every $u \in N_i$ with $i$ even and 
	$N^+(u) =\left( N_0\cup N_2 \cup \ldots N_{i+1}\right) \backslash w(u)$ for every $u \in N_i$ with $i$ odd.
	Assume there is such a vertex $u \in N_i$ ($i$ odd) for which $w(u) \not =v.$
	If $i>1$, then a vertex $z \in N_{i-1}$ satisfies $\ecc^+(z)=\max\{r-i+1,2\}<r$, which is a contradiction.
	For this, note that every vertex in $N_{ \le i+2}$ is at distance at most $2$ from $z,$
	as $\vc{zu}, \vc{uv}$ is a directed path towards $u$ and $N^+(z)= N_1 \cup N_3 \ldots \cup N_{i}.$
	If $i=1$, then any vertex $z \in N_{2}$ satisfies $\ecc^+(z)=\max\{r-2,2\}<r$ as $N^+(N^+(z))=N_{\le 4}$, which is again a contradiction.
	This implies that in an extremal digraph (satisfying the assumptions), for every $u \in N_i$ with $i$ odd we have $w(u)=v$, i.e. $N^+(u)=N_2\cup N_4 \cup \ldots N_{i+1}.$
	
	Such a digraph has size 
	\begin{align*}
	\lvert A(D) \rvert&= \sum_{ i \mbox{ even}}n_i(n_1+n_3+\ldots +n_{i+1}) + \sum_{ i \mbox{ odd}}n_i(n_2+n_4+\ldots +n_{i+1})\\
	&= n_0n_1+ \sum_{ i\ge 2 \mbox{ even}} n_i (n_o + n_{i-1}+n_{i+1})\\
	&=(n_e-1) n_o +\sum_{i\ge 0} n_i n_{i+1}.
	\end{align*}
	
	The term $(n_e-1) n_o$ is maximized when $\lvert (n_e-1)- n_o\rvert \le 1$.
	The maximum for the second term is determined in the following claim.
	\begin{claim}
		Let $n>r$ be two fixed integers.
		Let $n_1,n_2,\ldots, n_r \ge 1$ be $r$ integers for which $1+\sum_{i=1}^r n_i =n.$
		Then $n_1+\sum_{i=1}^{r-1} n_i n_{i+1}$ attains the maximum if and only if all $n_i$ are equal to $1$, except from at most $3$ consecutive values $n_j, n_{j+1},n_{j+2}$ satisfying $\lvert n_{j+1}+1-(n_j+n_{j+2})\rvert \le 1$ and $n_r=1$.
	\end{claim}
	
	\begin{claimproof}
		Let $n_0=1$ and let $x_i=n_i-1$ for every $i$.
		Let $x_e=\sum_{i\ge0} x_{2i}$ and $x_o=\sum_{i\ge 0} x_{2i+1}.$
		Then \begin{align*}
		n_1+2\sum_{i=1}^{r-1} n_i n_{i+1}&=r+2\sum_{i=1}^{r-1}x_i+x_{r}+\sum_{i=1}^{r-1}x_ix_{i+1}\\
		&\le r + 2(n-r-1)+x_ox_e\\
		&\le r + 2(n-r-1)+\left\lfloor \frac{(n-r-1)^2}4\right\rfloor
		\end{align*}
		Equality can only occur if every inequality is an equality.
		This implies that $x_r=0$ and if $j$ is the smallest index for which $x_j>0$, then one needs $x_{j+1}=x_o$ and $x_j+x_{j+2}=x_e$, as well as $\lvert x_e -x_o \rvert \le 1.$ Since $x_e -x_o=x_j+x_{j+2}-x_{j+1}=(n_j+n_{j+2})-(n_{j+1}+1)$, this is equivalent with the statement of the claim.
	\end{claimproof}
	When $r$ is even, we have $\lvert (n_e-1)- n_o\rvert=\lvert n_{j+1}+1-(n_j+n_{j+2})\rvert \le 1$ and so the condition is sufficient.
	When $r$ is odd, we have to add the condition $n_e \ge n_o$ as then $\lvert (n_e-1)- n_o\rvert \le 1$ will be satisfied once $\lvert n_{j+1}+1-(n_j+n_{j+2})\rvert \le 1$.	
\end{proof}

In Figure~\ref{fig:digraph_Dankelmann} an example of an extremal bipartite digraph with outradius $6$ is given, 
this is an example where the bipartition of the extremal digraph is not balanced (whenever $n$ is even).

\begin{figure}[h]
	\centering	
	\begin{tikzpicture}
	
	\definecolor{cv0}{rgb}{0.0,0.0,0.0}
	\definecolor{c}{rgb}{1.0,1.0,1.0}

	\Vertex[L=\hbox{$v_0$},x=2,y=-2]{v0}
	\Vertex[L=\hbox{$v_1$},x=3.5,y=1]{v1}
	\Vertex[L=\hbox{$v_2$},x=5,y=-2]{v2}
	\Vertex[L=\hbox{$a K_1$},x=6.5,y=1]{v3}
	\Vertex[L=\hbox{$\lceil \frac{n-5}2 \rceil K_1$},x=8,y=-2]{v4}
	\Vertex[L=\hbox{$\left(\lfloor \frac{n-3}2 \rfloor-a\right) K_1$},x=9.5,y=1]{v5}
	\Vertex[L=\hbox{$v_6$},x=11,y=-2]{v6}

	\Edge[lw=0.1cm,style={post, right}](v2)(v1)
	\Edge[lw=0.1cm,style={post, right}](v3)(v2)
	\Edge[lw=0.1cm,style={post, right}](v4)(v3)
	\Edge[lw=0.1cm,style={post, right}](v5)(v4)
	\Edge[lw=0.1cm,style={post, right}](v6)(v5)
	\Edge[lw=0.05cm,style={post, right}](v5)(v2)

	\Edge[lw=0.1cm,style={post, right}](v0)(v1)
	\Edge[lw=0.1cm,style={post, right}](v1)(v2)
	\Edge[lw=0.1cm,style={post, right}](v2)(v3)
	\Edge[lw=0.1cm,style={post, right}](v3)(v4)
	\Edge[lw=0.1cm,style={post, right}](v4)(v5)
	\Edge[lw=0.1cm,style={post, right}](v5)(v6)
	
	\Edge[lw=0.1cm,style={post, right}](v4)(v1)
	\Edge[lw=0.1cm,style={post, right}](v6)(v1)
	\Edge[lw=0.1cm,style={post, right}](v6)(v3)

	\end{tikzpicture}

	\caption{An extremal bipartite digraph of outradius $6$ and order $n$ maximizing the size}
	\label{fig:digraph_Dankelmann}
\end{figure}

\section{On bipartite biconnected digraphs with maximum size}\label{sec:bipbicDigraphs_maxSize}

For outradius $1$, the extremal digraph is a bidirected star.
Taking a vertex $v$ with $\ecc^+(v)=1$, we know all other vertices are in the other partition class and so no more than $2(n-1)$ arrows can be present.
When the outradius is $2$ and $n \ge 4$, the extremal digraph is the directed balanced complete bipartite digraph, which has size $\lfloor \frac{n^2}{2} \rfloor$.
So note that when restricting to bipartite digraphs, a larger radius does not imply a smaller size in this particular case.

When $\rad^+=3$, extremal digraphs can be obtained by, for every vertex $v$, removing one directed edge that starts in $v$, in such a way that the resulting digraph is still biconnected.
The bipartite subdigraphs of $D_{n,3,\lfloor \frac{n-4}{2} \rfloor}$ are among them, but are not the only extremal bipartite (biconnected) digraphs.

For outradius $r \ge 4$, we conjecture that the extremal digraphs are precisely of this form.
For example for $r=4$, the maximum size would be $\lfloor \frac{(n-2)^2}{2} \rfloor$ and being attained by the digraph in Figure~\ref{fig:Dnrs_bip}.

\begin{conj}\label{conj:extr_bicon_bip_digraph_rad+ge4}
	Among the bipartite biconnected digraphs of order $n$ and outradius $\rad^+=r \ge 4$, the maximum size is attained by $D_{n,r,\lfloor \frac{n}{2}-(r-1) \rfloor } \cap K_{ \lfloor \frac{n}{2} \rfloor, \lceil \frac{n}{2} \rceil }.$
	This being obtained by taking two blow-ups in $v_1$ and $w_1$ of the bipartite subdigraph of $D_{2r,r,1}$ with bipartition classes $V_1=\{v_i \mid i \equiv 1 \pmod 2\}\cup \{w_i \mid i \equiv 0 \pmod 2\}$ and $V_2=\{v_i \mid i \equiv 0 \pmod 2\}\cup \{w_i \mid i \equiv 1 \pmod 2\}.$
\end{conj}

\begin{figure}[h]
	\centering
	
	\scalebox{0.9}{
		\begin{tikzpicture}
		\definecolor{cv0}{rgb}{0.0,0.0,0.0}
		\definecolor{c}{rgb}{1.0,1.0,1.0}

		\Vertex[L=\hbox{$v_2$},x=13,y=0]{v1}
		\Vertex[L=\hbox{$v_3$},x=15,y=0]{v2}
		\Vertex[L=\hbox{$v_4$},x=16.5,y=0]{v3}
		\Vertex[L=\hbox{$w_2$},x=5,y=0]{w1}
		\Vertex[L=\hbox{$w_3$},x=3,y=0]{w2}
		\Vertex[L=\hbox{$w_4$},x=1.5,y=0]{w3}
		\Vertex[L=\hbox{$\lfloor \frac{n}{2}-3 \rfloor K_1$},x=7,y=0]{w0}
		
		\Vertex[L=\hbox{$\lceil \frac{n}{2}-3 \rceil K_1$},x=10cm,y=0.0cm]{v0}
		
		\Edge[lw=0.1cm,style={post, right}](v0)(v1)
		\Edge[lw=0.1cm,style={post, right}](v1)(v2)
		\Edge[lw=0.1cm,style={post, right}](v3)(v2)
		\Edge[lw=0.1cm,style={post, right}](w0)(v0)
		\Edge[lw=0.1cm,style={post, right}](v1)(v0)
		\Edge[lw=0.1cm,style={post, right}](v2)(v1)
		\Edge[lw=0.1cm,style={post, right}](v2)(v3)
		\Edge[lw=0.1cm,style={post, right}](v0)(w0)
		\Edge[lw=0.1cm,style={post, right}](w1)(w2)
		\Edge[lw=0.1cm,style={post, right}](w3)(w2)
		\Edge[lw=0.1cm,style={post, right}](w0)(w1)
		\Edge[lw=0.1cm,style={post, right}](w1)(w0)
		\Edge[lw=0.1cm,style={post, right}](w2)(w1)
		\Edge[lw=0.1cm,style={post, right}](w2)(w3)	
		
		\Edge[lw=0.1cm,style={post, bend right}](w2)(v0)
		\Edge[lw=0.1cm,style={post, bend left}](w3)(w0)
		
		\Edge[lw=0.1cm,style={post, bend right}](v2)(w0)
		
		\Edge[lw=0.1cm,style={post, bend left}](v3)(v0)
		

		\end{tikzpicture}}
	\caption{The bipartite digraph $D_{n,4,\lfloor \frac{n}{2}-3 \rfloor} \cap K_{ \lfloor \frac{n}{2} \rfloor, \lceil \frac{n}{2} \rceil}$ 	}
	\label{fig:Dnrs_bip}
\end{figure}

We prove Conjecture~\ref{conj:extr_bicon_bip_digraph_rad+ge4} for $n$ sufficiently large in terms of $r$, when $r$ is even.

\begin{lem}\label{lem:Kkksubdigraph}
	Let $D$ be a bipartite digraph $(V_1 \cup V_2, A)$ of order $n$ with average total degree at least being equal to $n-t$, where $n>9t$. 
	Then at least $\lvert V_1 \rvert - \frac n6$ vertices in $V_1$ have a total degree which is at least $2|V_2|-3t+1$ and at least $\lvert V_2 \rvert - \frac n6$ vertices in $V_2$ have a total degree which is at least $2|V_1|-3t+1$.
	Also $D$ contains a bidirected complete bipartite $K_{k,k}$ where $k \ge \frac n{18t},$ whose vertices have total degree being at least equal to $2|V_1|-3t+1$ or $2|V_2|-3t+1$ (depending on the bipartition class they belong).
\end{lem}

\begin{proof}
	The first part is true as the contrary would lead to a contradiction.
	E.g. when there are more than $\frac n6$ vertices in $V_1$ with total degree at most $2|V_2|-3t$, then 
	$\lvert A(D) \rvert < 2 \lvert V_1 \rvert \lvert V_2 \rvert - \frac n6 3t \le \frac{n(n-t)}{2}.$
	Note that $\lvert V_1 \rvert$ and $ \lvert V_2 \rvert$ are both equal to at least $\frac{n}{3}$ as otherwise $2 \lvert V_1 \rvert \lvert V_2 \rvert < \frac{n(n-t)}{2}.$
	
	Now let $S_1 \subset V_1$ and $S_2 \subset V_2$ be $\frac n6$ vertices with total degree at least $2|V_2|-3t+1$ or $2|V_1|-3t+1$ respectively.
	Now iterate the following procedure where one vertex of $S_1$ and one vertex of $S_2$ is taken in every step.
	Pick a vertex $v$ from $S_1$ and delete all $u \in S_2$ for which not both $\vc{vu}$ and $\vc{uv}$ belong to $A$. Next, do the same for $v \in S_2$ and deleting vertices in $S_1$ which are not connected in both directions with $v$.
	
	We can do this at least $\frac{n}{18t}$ times, as then at most $\frac{n}{18t}(3t-1)$ vertices of each set was discarded. 
	The $\frac{n}{9t}$ vertices (or more) at the end of this procedure form the desired $K_{k,k}$.
\end{proof}

\begin{lem}\label{lem:removevertexwithoutchangingdistances}
	Let $D$ be a bipartite digraph $(V_1 \cup V_2, A)$ with outradius $r\ge 4$, order $n$ and size at least $\frac{n(n-t)}{2}$ such that it contains a bidirected complete bipartite digraph $K_{k,k}$ for which all of its vertices have total degree being at least equal to $2|V_1|-3t+1$ or $2|V_2|-3t+1$ (depending on the bipartition class they belong).
	If $k>2(18t^2+3t)+1+\frac{tn}{2k}$, then for both choices of $i \in \{1,2\}$, there is a vertex $v \in K_{k,k}$ with $v\in V_i$ such that $D \backslash v$ has outradius $r$ and the distance between any $2$ vertices of $D \backslash v$ equals the distance between them in $D.$
\end{lem}

\begin{proof}
	Analogous to the proof of~\cite[Lem.~4.2]{cambie2019extremal} (which is a digraph version of Lemma~{lem3}), we can select sets $S_1 \subset V_1$ and $S_2 \subset V_2$, both containing at most $2\cdot(3t)^2+3t=18t^2+3t$ vertices.
	Hereby they are chosen so that for any $v \in K_{k,k} \backslash (S_1\cup S_2)$, the distance measure in $D \backslash v$ equals the restriction to $D \backslash v$ of the distance measure in $D.$ 
	Again there is at most one vertex such that the outradius might have increased.
	
	For any vertex $v \in K_{k,k} \backslash (S_1\cup S_2)$, the outradius of $D \backslash v$ is at least $r-1$. In case equality holds, there is exactly one vertex $w_v$ in $D \backslash K_{k,k}$ such that $d(w_v,v)=r$ and the out-eccentricity of $w_v$ as a vertex in $D \backslash v$ equals $r-1$.
	Let $T$ be the set of all vertices $v \in K_{k,k} \backslash (S_1\cup S_2)$ for which this happens.
	As for every $v \in T$, there is an associated $w_v$ which is not associated with another element from $T$, for which there is no edge from $w_v$ to any element of $K_{k,k}$ since $r \ge 4$.
	This implies that at least $k \lvert T \rvert $ possible arrows (i.e. arrows counted in the upper bound $2|V_1||V_2|$) are missing, which has to be at most $\frac{tn}{2}$ due to the lower bound on the size for $|A|$, i.e. $\lvert T \rvert \le \frac{tn}{2k}.$
	
	As $k>2(18t^2+3t)+1+\frac{tn}{2k} \ge \lvert S_1\cup S_2 \rvert +1+ \lvert T \rvert $, we can choose a vertex $v \in \left(V_i \cap K_{k,k}\right) \backslash (S_1 \cup S_2 \cup z^* \cup T).$ This vertex will satisfy all conditions of the Lemma.
\end{proof}

We now prove an analogue of Proposition~\ref{proplem_size}. The uniqueness part does not work when $r$ is odd, as then the blow-up of a $C_{r+1}$ does the job as well for example.
That is also why some more case analysis is needed here.

\begin{prop}\label{proplem_size2}
	Let $D=(V_1 \cup V_2,A)$ be a bipartite, biconnected digraph of order $n$ and outradius $r$, where $r \ge 4$ is even.
	Then for any vertex $v \in V_1$, the total degree $\deg(v)\le 2|V_2|-(r-2).$
	Equality can occur if and only if
	\begin{itemize}
		\item $\deg^-(v)=|V_2|$ and $\deg^+(v)=|V_2|-(r-2)$,
		\item there exist two disjoint directed paths $vu_1u_2\ldots u_{r}$ and $v w_2 w_3 \ldots w_r$ in $D$ with $d(v,u_r)=r$ and $d(v,w_r)=r-1$	and furthermore any two shortest paths from $v$ to $u_r$ and from $v$ to $w_r$ are disjoint.
	\end{itemize}
\end{prop}

\begin{proof}
	Let $\ecc^+(v)=r' \ge r$ and assume $vu_1u_2\ldots u_{r'}$ is a directed path in $D$ with $d(v,u_{r'}) =r'.$
	Let $i$ be the smallest index with $r'-r+1 \le i \le r'$ for which there is an arrow from $u_i$ to $v.$ Note that this one has to exist, or otherwise $\deg^+(v)\le |V_2|-\left(\lceil \frac{r'}{2} \rceil -1\right)$ and $\deg^-(v)\le |V_2|-\frac{r}2$ and so $\deg(v) <2|V_2|-(r-2).$\\
	Similarly as in Proposition~\ref{proplem_size}, there is a vertex $x$ with $d(u_{r'-r+1},v)+d(v,x)\ge d(u_{r'-r+1},x) \ge r$ for which $d(v,x) \ge r'-i.$\\
	The shortest path from $v$ to $x$ does not pass any $u_j$ for some $r' \ge j \ge r'-r+1$.
	So we have at least $\lceil \frac{r'-i}{2} \rceil -1$ vertices $w \in V_2 $, on the shortest path from $v$ to $x$, for which there is no directed edge from $v$ to $w$.\\
	First assume $r'>r$.
	There is no arrow from $v$ to any of the $\frac r2$ $u_j \in V_2$ where $r'-r+1 \le j \le r'$.
	Furthermore $v$ is not in the outneighbourhood of any of the $i+r-r'-1$ $v_j$ with $r'-r+1 \le j < i$. Since $i$ is odd, there are $\lceil \frac{r-r'+i}{2} \rceil -1$	of these belonging to $V_2$.
	We conclude that $\deg(v) \le 2|V_2|-\frac r2-\lceil \frac{r-r'+i}{2} \rceil-\lceil \frac{r'-i}{2} \rceil +2 \le 2|V_2|-r+2.$
	Furthermore, equality can only be attained when $r'$ and $i$ are both odd.
	But this implies that $r'-r+1$ is even and hence not equal to $i$. If $i<r$, we would know all possible arrows containing $v$ which are missing and so $\ecc^+(u_i)<r$.
	In the case that $i>r$, we have $\deg^+(v) \le |V_2|-\left(\lceil \frac{r'}{2} \rceil-1 \right)$ and $\deg^-(v)\le |V_2|-\left(\lceil \frac{r-r'+i}{2} \rceil -1 \right)$.
	Since $\lceil \frac{r'}{2} \rceil+\lceil \frac{r-r'+i}{2} \rceil > \frac{r'}{2} + \frac{r-r'+r}{2} =r$, we conclude again.
	
	So now we consider the case $r'=r$.
	If $1<i<r$, analogously as before by picking a vertex $x$ such that $d(u_i,x) \ge r$, again we conclude $\deg(v) \le 2|V_2|-(r-1)$.
	There are at least $\frac r2 -1$ vertices on the shortest path from $v$ to $x$ belonging to $V_2$ such that no arrow from $v$ ends in one of these vertices.
	Also there are no arrows from $u_j$ to $v$ for odd $j<i$ and from $v$ to $u_j$ for $j\ge i$.
	In the remaining case with $i=1$, to ensure $\ecc^+(u_1)\ge r$, there has to be a shortest path from $v$ to $x$ which is disjoint from $vu_1u_2 \ldots u_r$, where $d(v,x) \ge r-1.$
\end{proof}

\begin{thr}\label{thr:bipbicDigraphMaxSize}
	Let $r \ge 4$ be even. Then there exists a value $n_1(r)$ such that for every $ n \ge n_1$, such that for any bipartite biconnected digraph $D$ of order $n$ and outradius $r$,
	$$\lvert A(D) \rvert \le \left\lvert A\left(D_{n,r,\lfloor \frac{n}{2}-(r-1) \rfloor } \cap K_{ \lfloor \frac{n}{2} \rfloor, \lceil \frac{n}{2} \rceil }\right) \right\rvert.$$
	Furthermore the digraph $D_{n,r,\lfloor \frac{n}{2}-(r-1) \rfloor } \cap K_{ \lfloor \frac{n}{2} \rfloor, \lceil \frac{n}{2} \rceil }$ is the unique bipartite biconnected digraph maximizing the size given $n$ and $r$.
\end{thr}

\begin{proof}
	Let $t=2(r-2)$ and $n_0(r)=18t(45t^2+6t+1)+2$. 
	Assume that $n \ge n_0$ and that $D=(V_1 \cup V_2,A)$ satisfies $\lvert A(D) \rvert \ge \lvert A(D_{n,r,\lfloor \frac{n}{2}-(r-1) \rfloor } \cap K_{ \lfloor \frac{n}{2} \rfloor, \lceil \frac{n}{2} \rceil }) \rvert \ge \frac{n(n-t)}{2}+1$.
	Note that $ \frac{n}{18t} > 2(18t^2+3t) +1+ 9t^2 \ge 2(18t^2+3t) +1+\frac{ tn}{2k}$.
	So by Lemma~\ref{lem:Kkksubdigraph} and Lemma~\ref{lem:removevertexwithoutchangingdistances}, we have a vertex $v$ in the largest bipartition class such that $D_1=D \backslash v$ is still biconnected and also has outradius $r.$
	Repeating this, gives a sequence $D_0=D,D_1,D_2, \ldots, D_{n-n_0}$ of digraphs all being biconnected and of outradius $r$. 
	Since
	\begin{equation}\label{eq:difference_size_Di}
	\lvert A(D_i) \rvert-\lvert A(D_{i+1})\rvert =\deg_{D_i}(v)\le 2 \left \lfloor \frac{n-i}2 \right\rfloor -(r-2)
	\end{equation} by Proposition~\ref{proplem_size2}, we see inductively that $\lvert A(D_i)\rvert \le \frac{(n-i)(n-i-t)}2$ does hold for every $0 \le i \le n-n_0$.
	Note for this that two consecutive numbers cannot be both even.
	Take $n_1(r)=5rn_0>n_0+5(r-2)n_0$.
	
	Note that $\lvert A(D_{n-n_0}) \rvert \le \binom{n_0}{2}$.
	This implies that if there are more than $(r-2)n_0$ values $i$ for which there is no equality in Equation~\eqref{eq:difference_size_Di}, $\lvert A(D) \rvert \ge \frac{n(n-t)}{2}$ was impossible.
	In particular, this implies that the partition classes cannot differ by more than $(r-2)n_0$.
	Furthermore, once they are balanced, we can assume we discard alternately a vertex from $V_1$ and one of $V_2$ to get the sequence of digraphs $(D_i)_i.$
	By the pigeon hole principle, we note that there must be two consecutive vertices discarded for which Equation~\eqref{eq:difference_size_Di} does hold and we can assume they did belong to the same digraph $K_{k,k},$ i.e. are neighbours.
	
	Let $v \in V_1$ and $v' \in V_2$ be two such neighbours and let $D_i$ be the digraph with largest index containing both $v$ and $v'.$
	By Proposition~\ref{proplem_size2}, we know the structure of $D_i$ partially.
	Since $v$ and $v'$ are neighbours, $N^+(v)$ cannot contain any of $\{u_4,u_6, \ldots u_r\}$ and $\{w_5,w_7, \ldots, w_{r-1}\}$.
	Since $\deg(v')=2|V_1|-(r-2)$ it has $u_2$ or $w_3$ as a neighbour (and not both).
	
	Assume the first case, i.e. $\vc{v'w_3} \in A$. In this case we can assume $v'=w_2$ actually.
	Since the same structure of Proposition~\ref{proplem_size2} does hold for $v'$, we have $v_2v' \in A$ and hence $\ecc^+(v_2)\le r-1$, contradiction. In Figure~\ref{fig:Dnrs_bip_partial_case1} this is illustrated in the case $r=4$, with $V'_2 \subset V_2$ and $V'_1 \subset V_1$ being the remaining vertices of both sets of vertices.
	
	\begin{figure}[h]
		\centering
		
		\begin{tikzpicture}
		\definecolor{cv0}{rgb}{0.0,0.0,0.0}
		\definecolor{c}{rgb}{1.0,1.0,1.0}
		\Vertex[L=\hbox{$V'_1$},x=6.5,y=-1.5]{w5}
		\Vertex[L=\hbox{$V'_2$},x=5.5,y=1.5]{v5}
		\Vertex[L=\hbox{$v_2$},x=13,y=0]{v1}
		\Vertex[L=\hbox{$v_3$},x=15,y=0]{v2}
		\Vertex[L=\hbox{$v_4$},x=16.5,y=0]{v3}
		\Vertex[L=\hbox{$v'$},x=5,y=0]{w1}
		\Vertex[L=\hbox{$w_3$},x=3,y=0]{w2}
		\Vertex[L=\hbox{$w_4$},x=1.5,y=0]{w3}
		\Vertex[L=\hbox{$v$},x=7,y=0]{w0}
		
		\Vertex[L=\hbox{$v_1$},x=10cm,y=0.0cm]{v0}
		
		\Edge[lw=0.1cm,style={post, right}](v0)(v1)
		\Edge[lw=0.1cm,style={post, right}](v1)(v2)
		\Edge[lw=0.1cm,style={post, right}](w0)(v0)
		\Edge[lw=0.1cm,style={post, right}](v2)(v3)
		\Edge[lw=0.1cm,style={post, right}](v0)(w0)
		\Edge[lw=0.1cm,style={post, right}](w1)(w2)
		\Edge[lw=0.1cm,style={post, right}](w0)(v5)
		\Edge[lw=0.1cm,style={post, right}](v5)(w0)
		\Edge[lw=0.1cm,style={post, right}](w1)(w5)
		\Edge[lw=0.1cm,style={post, right}](w5)(w1)
		\Edge[lw=0.1cm,style={post, right}](w0)(w1)
		\Edge[lw=0.1cm,style={post, right}](w1)(w0)
		\Edge[lw=0.1cm,style={post, right}](w2)(w1)
		\Edge[lw=0.1cm,style={post, right}](w2)(w3)	
		
		\Edge[lw=0.1cm,style={post, bend left}](w3)(w0)
		
		\Edge[lw=0.1cm,style={post, bend right}](v2)(w0)
		\Edge[lw=0.1cm,style={post, bend left=12.5}](v1)(w1)
		\Edge[lw=0.1cm,style={post, bend left=15}](v3)(w1)
		
		%
		%
		
		\end{tikzpicture}
		\caption{A possible substructure of a digraph with the maximum size, when $r=4$}
		\label{fig:Dnrs_bip_partial_case1}
	\end{figure}
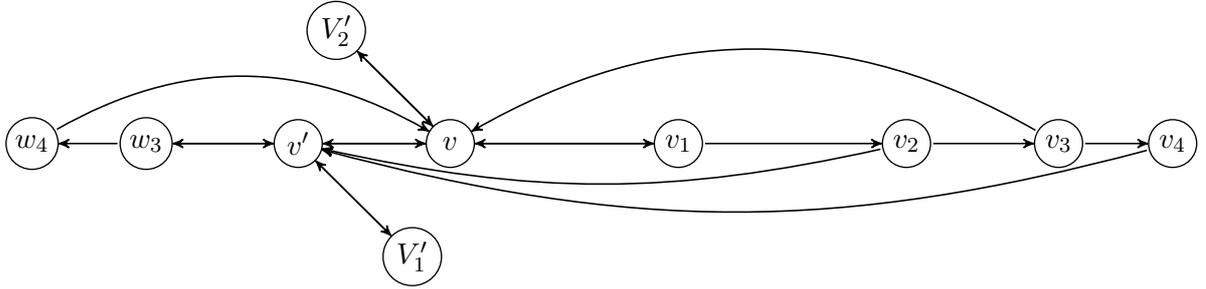
	So we can go on with the case that $v'=v_1$.
	We know some of the arrows already by using Proposition~\ref{proplem_size2} for both $v$ and $v'$.
	This is sketched in Figure~\ref{fig:Dnrs_bip_partial}.
	There can be no arrow from a $v_i$ with $i\ge 3$ odd to some $w_j$ with $j \ge 2$ (odd as well) as then $\ecc^+(v_i)<r$.
	An arrow from $v_i$ with $i$ even to $w_j$ with $r>j \ge 2$ (even) would also imply that $\ecc^+(v_i)<r$.
	Last, if $\vc{v_i w_r} \in A$ for $i$ even, it would imply that $\ecc^+(v_{i-1})<r$.
	The same holds in the opposite direction, with $w_i$ and $v_j.$
	Also for every vertex $u'$ in $V'_2$, its outneighbourhood can only contain $v,V'_1$ and one of $v_2$ or $w_3$. 
	Similar one gets a bound for $V'_1$.
	All of this results in an upper bound on the size of the digraph which is exactly equal to the size of the conjectured extremal digraph.
	Note that no $u'$ in $V'_2$ can be connected to $w_3$ in a digraph attaining the upper bound, as then again $\ecc^+(v_2)<r$.
	So every $u' \in V'_2$ will have the same set of neighbours as $v'$ and so do vertices in $V'_1$ and $v$.
	Last, we note that the maximum is only attained if $V'_1$ and $V'_2$ are balanced and so the extremal digraph is exactly as stated.
\end{proof}

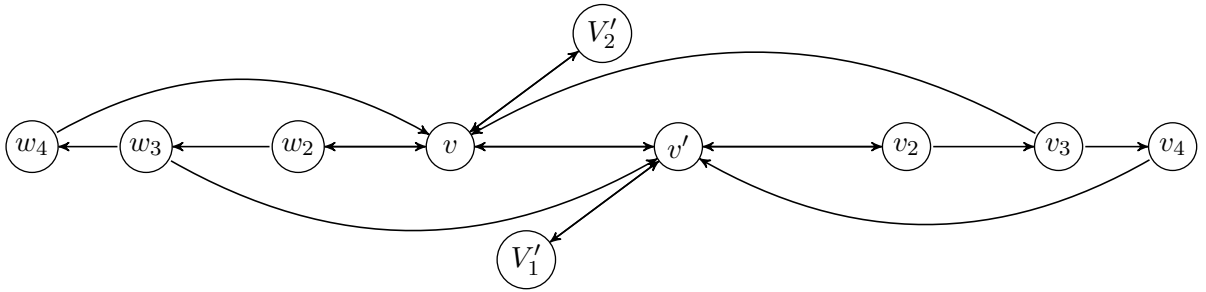
\begin{figure}[h]
	\centering

	\begin{tikzpicture}
	\definecolor{cv0}{rgb}{0.0,0.0,0.0}
	\definecolor{c}{rgb}{1.0,1.0,1.0}

	\Vertex[L=\hbox{$v_2$},x=13,y=0]{v1}
	\Vertex[L=\hbox{$v_3$},x=15,y=0]{v2}
	\Vertex[L=\hbox{$v_4$},x=16.5,y=0]{v3}
	\Vertex[L=\hbox{$w_2$},x=5,y=0]{w1}
	\Vertex[L=\hbox{$w_3$},x=3,y=0]{w2}
	\Vertex[L=\hbox{$w_4$},x=1.5,y=0]{w3}
	\Vertex[L=\hbox{$v$},x=7,y=0]{w0}
	
	\Vertex[L=\hbox{$V'_1$},x=8,y=-1.5]{w5}
	\Vertex[L=\hbox{$V'_2$},x=9,y=1.5]{v5}
	
	\Vertex[L=\hbox{$v'$},x=10cm,y=0.0cm]{v0}
	
	\Edge[lw=0.1cm,style={post, right}](w0)(v5)
	\Edge[lw=0.1cm,style={post, right}](v5)(w0)
	\Edge[lw=0.1cm,style={post, right}](v0)(w5)
	\Edge[lw=0.1cm,style={post, right}](w5)(v0)
	
	\Edge[lw=0.1cm,style={post, right}](v0)(v1)
	\Edge[lw=0.1cm,style={post, right}](v1)(v2)
	\Edge[lw=0.1cm,style={post, right}](w0)(v0)
	\Edge[lw=0.1cm,style={post, right}](v1)(v0)
	\Edge[lw=0.1cm,style={post, right}](v2)(v3)
	\Edge[lw=0.1cm,style={post, right}](v0)(w0)
	\Edge[lw=0.1cm,style={post, right}](w1)(w2)
	\Edge[lw=0.1cm,style={post, right}](w0)(w1)
	\Edge[lw=0.1cm,style={post, right}](w1)(w0)
	\Edge[lw=0.1cm,style={post, right}](w2)(w3)	
	
	\Edge[lw=0.1cm,style={post, bend right}](w2)(v0)
	\Edge[lw=0.1cm,style={post, bend left}](w3)(w0)
	
	\Edge[lw=0.1cm,style={post, bend right}](v2)(w0)
	
	\Edge[lw=0.1cm,style={post, bend left}](v3)(v0)

	\end{tikzpicture}
	\caption{A second possible substructure of a digraph with the maximum size, when $r=4$}
	\label{fig:Dnrs_bip_partial}
\end{figure}

\begin{remark}
	It is not hard to see that the same proof gives the asymptotic formula for the maximum size given outradius $r$ for $r$ odd as well, since Proposition~\ref{proplem_size2} still applies. 
	That is, for every $r \ge 4$ and any biconnected bipartite digraph $D$ one has $\lvert A(D) \rvert \le \frac{n^2}{2}-(r-2)n+O_r(1).$
\end{remark}

\section{Conclusion}\label{conc}

When considering order and diameter, as a corollary of folklore results~\cite{Ore68,P84} the extremal graphs or digraphs attaining the minimum average distance do maximize the size, but the set of extremal graphs are not equal. We investigated if there are similar relationships between the extremal graphs and digraphs attaining the minimum average distance or maximum size given order and radius. It turns out that the story is slightly harder in this case, as one can conclude from e.g. $Q_3$ and the examples in Figure~\ref{fig:digraph_rad2.5_SmallWm}.
Nevertheless, for $n$ large enough with respect to $r$, the extremal graphs are $G_{n,r,s}$ for both questions (by \cite{VZ67} and \cite[Thr.~1.2]{cambie2019extremal}) and so do the digraphs $D_{n,r,s}$ for the outradius $r$ (by \cite[Thr.~1.4]{cambie2019extremal} and Theorem~\ref{mainVz_bicDi}).
For digraphs given order and radius $\rad=r$, the story is again different. 
Taking a blow-up of vertex $2$ or $4$ by a $K_{n-5}$ in the left digraph in Figure~\ref{fig:digraph_rad_smallWm} results in a graph whose total distance is $7$ less than the $\overline{ \Gamma}_{n,6,3,n-6}$. So in this case, there would be no relation between the two sets of extremal digraphs if Conjecture~\ref{conj:VZdi_rad} is true.

While investigating the maximum size of biconnected digraphs given order and (out)radius, next to our main result Theorem~\ref{mainVz_bicDi}, we also wondered about the maximum size of bipartite digraphs given order and outradius.
We characterized the extremal bipartite digraphs in the general case. This has been done in Theorem~\ref{thr:bipdigraph_maxSize}.
In contrast, for biconnected bipartite digraphs, the extremal digraph seems to be unique in certain cases, which we proved for $r\ge 4$ even and $n$ sufficiently large in terms of $r$ (Theorem~\ref{thr:bipbicDigraphMaxSize}).

Overall, we conjectured the answers for~\cite[Prob.2 \& Prob.3]{D15} and proved these in certain cases. A few questions and conjectures are still open at this point.
Perhaps the most important conjecture (even while conjecture~\ref{conj:VZdi_rad} may be considered equally elegant) in this paper is conjecture~\ref{VZdi}, which once fully proven solves \cite[Problem 2]{D15}. 

\paragraph{Open access statement.} For the purpose of open access,
a CC BY public copyright license is applied
to any Author Accepted Manuscript (AAM)
arising from this submission.

\bibliographystyle{abbrv}
\bibliography{MaxMuD}

\end{document}